\documentclass[a4wide]{article} 

\usepackage[notcite, notref]{}
  
\usepackage[
a4paper,
left=35mm,
right=35mm, 
top=14mm,
bottom=17mm]{geometry}

\usepackage{todonotes}
\usepackage[english]{babel} 
\usepackage[utf8]{inputenc}
\usepackage[T1]{fontenc}
\usepackage{verbatim} 
\usepackage{float}

\usepackage{amsmath}
\usepackage{amsfonts}
\usepackage{amssymb}
\usepackage{amsthm}

\usepackage{mathrsfs}
\usepackage{mathtools}
\usepackage{euscript}   
\usepackage[scr=boondoxo,scrscaled=1.05]{mathalfa} 

\usepackage{pifont} 
\usepackage{tikz-cd} 

\usepackage{subcaption}

\usepackage{listings}
\usepackage{xcolor}

\usepackage[colorlinks=true, pagebackref=true]{hyperref}
\hypersetup{linkcolor=black, citecolor=black, urlcolor=black}
\usepackage{comment} 

\usepackage{cancel}

\newtheorem{theorem}{Theorem}[section]

\newtheorem{conjecture}[theorem]{Conjecture}

\newtheorem{proposition}[theorem]{Proposition}
\newtheorem{lemma}[theorem]{Lemma}
\newtheorem{corollary}[theorem]{Corollary}

\newtheorem{remark}[theorem]{Remark}
\newtheorem{definition}[theorem]{Definition}
\newtheorem{example}[theorem]{Example}

\usepackage[all]{xy}

\usepackage{bbm}  

\newcommand\EE{\mathbb{E}}

\newcommand\NN{\mathbb{N}}

\newcommand\RR{\mathbb{R}}

\newcommand\one{\mathbbm{1}} 
\newcommand\nub{\nu_{\scalebox{0.4}{\bf +}}}
\newcommand\nua{\nu_{\scalebox{0.7}{\bf -}}}
\newcommand\nb{n_{\scalebox{0.4}{\bf +}}}
\newcommand\na{n_{\scalebox{0.7}{\bf -}}}

\newcommand{\G}{\mathcal{G}}

\newcommand{\cC}{\mathcal{C}}

\newcommand{\cT}{\mathcal{T}}

\newcommand{\rmT}{T}

\newcommand{\tb}{{\tilde b}}

\newcommand{\sB}{\mathsf{B}}

\newcommand{\bari}{{\mathsf{b}}}

\newcommand\eps{{\varepsilon}}
\newcommand\vep{{\varepsilon}}

\DeclareMathOperator\Cov{Cov}
\DeclareMathOperator{\sign}{sign}

\newcommand{\toe}{\;\mathop{\longrightarrow}_{\eps\to0}\;}

\allowdisplaybreaks
\title{Continuous auction models}
\author{Gioia Carinci, Pablo A. Ferrari,\\ Chiara Franceschini,	Nicola Manelli}

\begin{document}

\maketitle

\paragraph{Abstract} 
We study a discrete-time auction model in which multiple sellers update their bids according to their performance in the preceding round. Bidders aim to maximize their profits and adjust their bids solely based on their most recent outcome (myopic behavior). In each round, the auctioneer purchases the lowest-priced $p$-fraction of the total quantity offered by the bidders. We find a system of differential equations governing the macroscopic dynamics and derive it as a scaling limit of the microscopic model. We find an explicit solution for the max-price evolution $q_t$ and show that, in the long run, bidders coordinate, i.e., their bids converge to a common value depending only on their initial distribution and the fraction $p$. For Poisson-distributed initial bids, we establish hydrodynamic limits for the empirical bid distribution and the max-price trajectory and conjecture the corresponding Gaussian fluctuations for $q_t$. Finally, we generalize the model to allow for heterogeneous bid-update velocities: in this case, the max-price velocity becomes proportional to the harmonic mean of the update velocities of bidders at the max-price.



\tableofcontents

\section{Introduction}


Auctions are a fundamental tool for selling different kinds of goods and resources, ranging from traditional commodity markets to modern online platforms. Their economic foundations are extensively studied in auction theory,  see Krishna \cite{krishna2009auction}.
From a mathematical point of view, several auction models have been proposed in the literature, each developed to capture different aspects of the bidding process, market structure and participant behavior.
When the number of participants is large, however, a detailed game-theoretic description becomes difficult and one is naturally led to a statistical description of the collective behavior of bidders.
In this framework, the interactions among many agents are viewed as collisions in a many-particle system, allowing auction dynamics to be studied using methods from kinetic theory and statistical physics, see Pareschi and Toscani \cite{pareschi2013interacting}.
Another well known example of trading mechanism is the limit order book dynamics. Such models provide a framework for buyers and sellers to interact with each other in financial markets, see \cite{jain2024limit} for a review. In  \cite{cont2025mathematical} Cont, Degond and Xuan propose a general framework where the dynamic is decomposed into an incoming order flow, represented by a spatial point process and the market clearing, modeled by a deterministic operator acting on the distributions of buy and sell orders. 


In this paper we consider a model proposed by Pinasco Saintier Kind (PSK) \cite{psk24} where there is a unique auctioneer, multiple seller agents and successive auction rounds. Let $n$ be a positive natural number, $\gamma>0$ a small parameter and $0\le b_1\dots\le b_n\le1$ where the bids belong to the discrete set, $\{0,\gamma, 2\gamma,\dots,1-\gamma,1\}$. At the first round, seller $i$ offers $1/n$ of the total mass $1$ at bid $b_i$. A fixed fraction $p=k/n$ of the total mass is awarded to bidders $1, \dots, np$, so that after the first round, there are $np$ winners and $n(1-p)$ losers. For the second round, each bidder updates her bid only based on the outcome of the previous round, win/loss, following a myopic strategy. Winners of the first round increase their bid by $\gamma$ and losers  decrease theirs by $\gamma$; boundary conditions are fixed so that bids remain in $[0,1]$. Rounds occur at times multiples of $\gamma$. In the long run, bidders tend to converge to a common bid value, which depends on both the initial distribution of bids and the value of $p$. See Figure \ref{discreto}.
\begin{figure}[th]
  \centering
  \includegraphics[width= .7\textwidth]{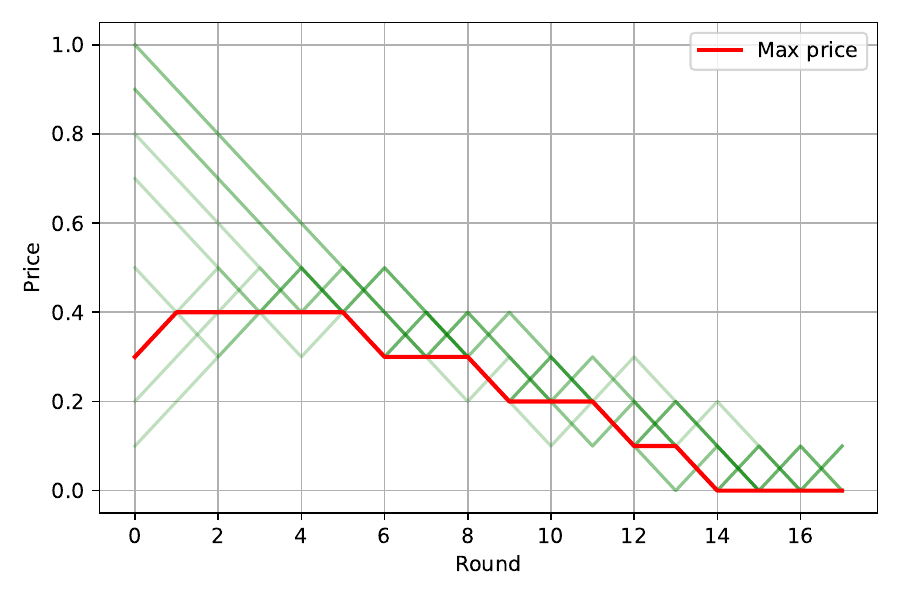}
  \caption{PSK model with $p=0.3$, $n=10$ and $\gamma=0.1$. The green lines represent bid evolution, darker green lines represent multiple bids at the same value; there are double bids at $1$ and $0.9$. The red line is the max-price evolution. At each round, winners increase their bids by $\gamma$ and losers decrease their bids by the same value.}
  \label{discreto}
\end{figure}
The papers \cite{smkp21, psk24} consider this system with the addition of Gaussian noise to the bid updating, propose a scaling limit by taking $\gamma$ to zero to obtain a system of ordinary differential equations for the empirical measure of the bids, show that there is a phase transition at $p=1/2$ and perform several simulations confirming their results and conjectures. The paper \cite{smkp21} considers the interesting case of $p$ varying in time, whose inverse is called competition level.

We introduce a continuous version of the PSK model, with an arbitrary initial bid distribution $\mu$, a probability measure on $[0,1]$ and denote by $q_t$    the max-price paid by the buyer at time $t$, which is also  the largest winner bid at time $t$. The function $q_t$, $t\ge 0$ is defined as the solution of a system of ordinary differential equations for $q_t$, $w_t$ and $\ell_t$ as a function of $\mu$, where $w_t$ and $\ell_t$ are the winning and losing mass bidding $q_t$ at time $t$. We provide an explicit solution in terms of the barycenter of the bids coinciding with the max-price. See Figure \ref{sin}. The resulting bid measure at time $t$, denoted $\mu\mathcal{T}_t$, puts mass $(w_t + \ell_t)$ at $q_t$, shifts the set of initial winners that have not yet attained $q_t$ by $t$ and shifts the set of initial losers that have not yet attained $q_t$ by $-t$. Figure \ref{ball} shows the continuous model with an initial sum of delta measures.
\begin{figure}[th]
  \centering
\includegraphics[scale=0.4]{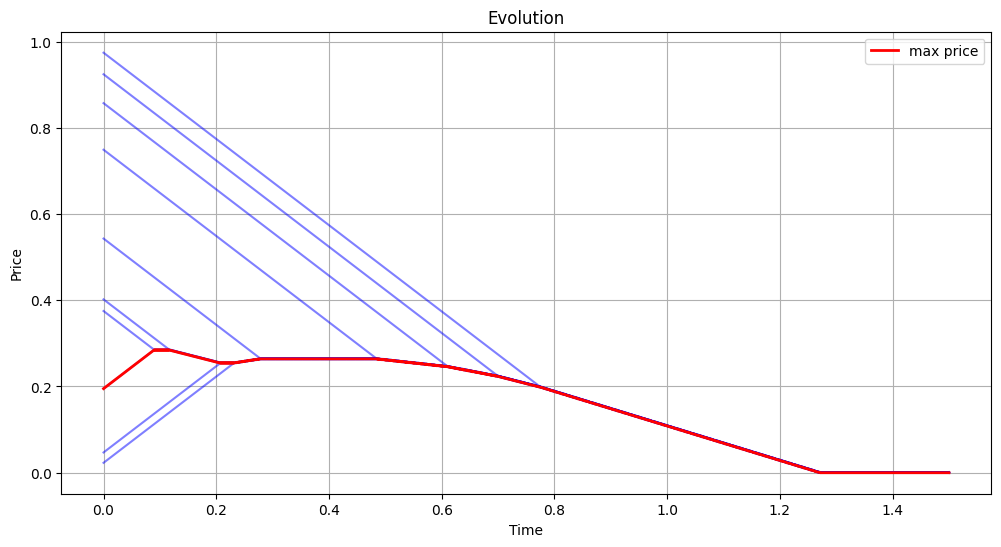} 
  \caption{The continuous model with initial sum of $10$ delta measures and $p=0.3$. Winners travel ballistically  at speed 1 and losers travel down at the same velocity. Both modify and adopt the velocity of the max-price upon colliding with it. The max-price travels ballistically between collisions.}
  \label{ball}
\end{figure}

We show analogous theorems for the discrete PSK model. In particular, we show that $q^\gamma_t$ remains at a distance of less than $2\gamma$ from the barycenter of those bids which have hit the max-price within time $t\gamma$. We also show that when the initial bids are concentrated on three sites, the motion of the max-price is periodic.   
\begin{figure}[th]
  \centering
  \includegraphics[width= .7\textwidth]{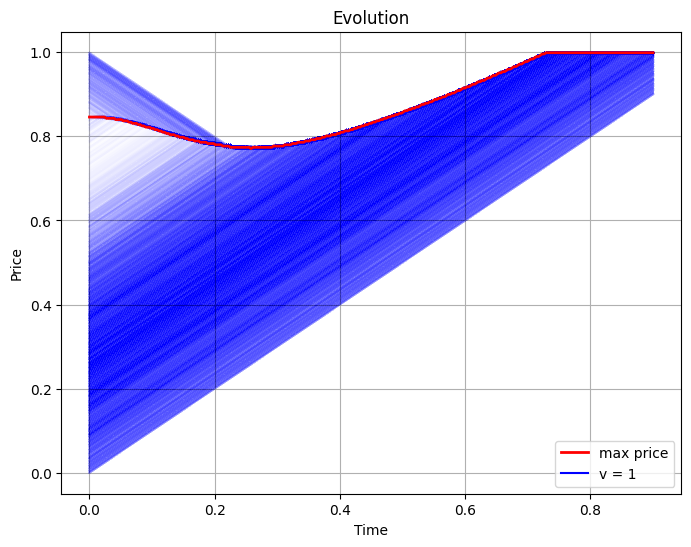}
  \caption{The continuous auction model with non homogeneous initial delta measures. The initial density is a sinusoidal curve; $p=0.9$. Winners travel ballistically upwards at speed $1$ and losers travel downwards at the same speed. When a bid collides with the max-price, the red line in the picture, they coalesce and modify the speed of the max-price. When the max-price arrives at one of the extremes of the interval, it remains there.}
  \label{sin}
\end{figure}
We show that the continuous model starting with a finite sum of delta measures can be approached by a family of PSK models indexed by the parameter $\gamma$, as $\gamma$ goes to zero. This convergence justifies the use of the continuous model, which is easier to deal with. After that we consider a family of non-homogeneous bid Poisson processes $B^\vep$ on $[0,1]$ with mean measure $\vep^{-1}\mu$ and show that the bid empirical measure $\mu^\vep\cT_t$ converges as $\vep\to0$ to $\mu \cT_t$. In particular, the empirical max-price trajectory $q^\vep_t$ converges almost surely to the expected one $q_t$. We perform the diffusive rescaling of the max-price and conjecture that it converges to a Gaussian process with non trivial covariances that are computed explicitly.

A striking property of this model is that when the updating bid policy is non-homogeneous, the max-price velocity is proportional to the harmonic mean of the velocities of the bids coinciding with the max-price. In the discrete case, agent $i$ has an associated velocity $v_i$, a natural number, and after each round increases/decreases her bid by $v_i\gamma$. We show that the max-price velocity is given by $  \gamma h \cdot (w-\ell)/(w+\ell)$, where $h=(w+\ell)/\sum_i(1/v_i)$ and the sum runs over the $i$ for which $b_i$ is close to $q$. See Figure \ref{2vel}. We show that this is the case when there are only two bid-velocities, and for multiple bids, assuming that the motion is periodic. In the continuous case with an initial sum of a finite number of delta measures, we describe the system of ordinary differential equations governing the evolution. Since the velocities vary, it is possible that bids hitting the max-price get out of it, because their velocity is not high enough to follow the max-price and we show how this separation occurs. See Figure \ref{2velpas} in Section \ref{s62}.
\begin{figure}[th]
  \centering
  \includegraphics[width= .6\textwidth, trim=0 8mm  0 0, clip]{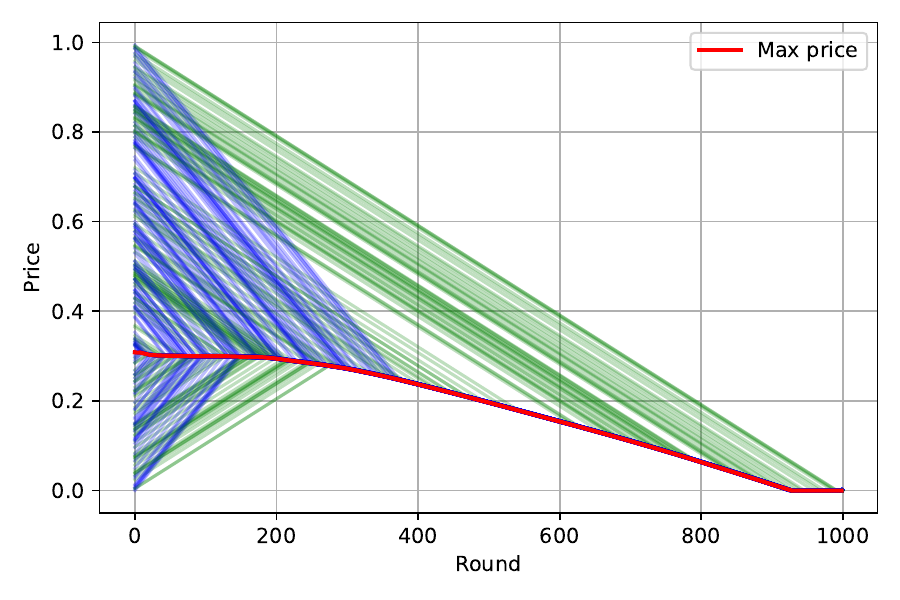}
  \caption{Two different velocities $\nu_{-}=1$ and $\nu_{+}=2$ with $p=0.3$ and random initial measure for both fast and slow bids.}
  \label{2vel}
\end{figure}

The paper is organized as follows. We start by introducing the continuous model in Section \ref{Cont}. We define it as the solution of a system of ordinary differential equations and give an explicit formula for the solution. In Section \ref{discretemodel}, we describe the corresponding discrete PSK model and, following the approach of the previous section, we provide an approximate solution of first order  in $\gamma$ for the system of discrete equations. We show that the discrete process is locally periodic when all bids belong to a $\gamma$ neighbourhood of the max-price. In Section \ref{s4}, we show that the discrete model converges to the continuous one as $\gamma$ tends to zero. In Section \ref{s5} we consider the continuous process with initial condition $\sum_b\delta_b$, where the $\delta_b$ are delta measures and the $b$ are the points of a Poisson process. We prove the hydrodynamic limits for the empirical measure of bids at each time $t$ and for the empirical max-price. We then diffusively rescale  the trajectory of the max-price around the mean trajectory to obtain a Gaussian diffusion with non trivial covariances. Finally, in Section \ref{s6} we generalize the PSK model by allowing each bidder to have an update policy, called velocity. We prove that, in the case of two velocities, the max-price velocity is a multiple of the harmonic mean of  the velocities of the bids  adhering to the max-price. We establish this result in the continuous model and show that the multi-velocity discrete model converges to the continuous one.

\section{Continuous model}\label{Cont}
In this section, we introduce a continuous space-time version of the PSK model by defining its dynamics through system of ordinary differential equations governing the joint evolution of the bids and the max-price.

Let $p\in[0,1]$ and $\mu$ be a probability measure on $[0,1]$ and let the cumulative function $F$ be defined by $F(x) := \int_0^x \mu(dx')$. Denote by  $q_0$ the max-price at time $0$, given by the $p$-quantile of the bids at time $0$, $q_0= F^{-1}(p):=\inf\{x\ge0: F(x)\ge p\}$. The heuristic of the evolution is as follows. Bids strictly less than $q_t$, called winners, travel ballistically at velocity $1$, those strictly larger than $q_t$, losers, travel ballistically at speed $-1$. Denoting by $w_t$ and by $\ell_t$ the amount of winners and losers at the max-price at time $t$, the velocity $u_t$ of the max-price is given by the average $(w_t-\ell_t)/(w_t+\ell_t)$.
\subsection{Evolution}
We define a system of ordinary differential equations for the evolution of $q_t$ as follows. 

Let $\mu$ be a probability on $[0,1]$ and $F$ the corresponding cumulative probability distribution function. For each $p\in (0,1)$, define the system
\begin{align}
  \frac{dq_t}{dt}
  &= u_t \,  \cdot \one\{q_t\in(0,1)\}  
    \label{wt}\\
   w_t&
        = p - F((q_t - t)^-)\label{wt1}\\
  \ell_t&
          = F(q_t + t) -p\label{lt1}\\\
  u_t&=
  \frac{w_t-\ell_t}{w_t+\ell_t}\;\label{ut}\\
 q_0&=  F^{-1}(p)\qquad \text{(initial condition)}
 \label{qb0}
\end{align}

Eventually, all the bids are absorbed in the max-price. This absorption can occur either before or after $q_t$ hits the ceiling ($q_t = 1$) or the floor ($q_t = 0$) depending on the value of $p$ and the initial mass distribution. When this happens $u_t$ vanishes and $q_t$ remains constant as time evolves.

The cumulative distribution $F_t$ at time $t$ is a function of $q_t$ and the initial $F$,
\begin{align}
  \label{17}
  F_t(x) := \bigl[(1-F(x+t))^+\, + (w_t+\ell_t)\bigr]\, \one\{x\ge q_t\} + F(x-t) \one\{x< q_t\}
\end{align}
This coincides with the equation proposed by Pinasco, Saintier and Kind \cite{psk24}.

\paragraph{\bf Existence and uniqueness of solutions}
Let  $G:[0,1]\to \mathbb R^+$ and $\G_t:[0,1]\to \mathbb R^+$ be the functions defined by
\begin{align}
  G(x)&:=\int_0^x F(z)\, dz  \label{Gdef}\\
  \G_t(x)&:= G(x+t) - G(x-t) =\int_{x-t}^{x+t} F(z)\, dz\label{tG1}
\end{align} 
Since $F$ is non-decreasing, also  $\G_t$ is non decreasing even if it is not necessarily strictly increasing.
Both $G$  and $\G_t$ are continuous and semi-differentiable, with right and left derivatives given by:
 \begin{align}
  \partial_+ G(x)&:= \lim_{\eps \searrow 0}\frac{G(x+\eps)-G(x)}{\eps}=F(x)\\
    \partial_- G(x)&:= \lim_{\eps \searrow 0}\frac{G(x)-G(x-\eps)}{\eps}=F(x^-)
 \end{align}
 and
 \begin{align}
  \partial_+ \G_t(x)&:= \lim_{\eps \searrow 0}\frac{\G_t(x+\eps)-\G_t(x)}{\eps}\\
  &= \partial_+ G(x+t)-  \partial_+ G(x-t)\nonumber
  \\
  &=F(x+t)-F(x-t)\nonumber
 \end{align}
 \begin{align}
  \partial_- \G_t(x)&:= \lim_{\eps \searrow 0}\frac{\G_t(x)-\G_t(x-\eps)}{\eps}\\
 &= \partial_- G(x+t)-  \partial_- G(x-t)\nonumber
  \\
  &=F((x+t)^-)-F((x-t)^-)\nonumber
 \end{align}
\begin{lemma} For all fixed $t\ge 0$, the problem
\begin{align}
  \label{35}
  \G_t(x_t)= 2pt, \qquad \text{and} \qquad x_t \in [q_0-t, q_0+t]
\end{align}
admits a unique solution $x_t:=\G_t^{-1}(2pt)$. 
\end{lemma}

\begin{proof}
We have $F(q_0)=p$,  $F(x)< p$ for $x< q_0$ and $F(x)\ge p$ for $x\ge q_0$. As a consequence,
  \begin{equation*}
\G_t(q_0-t) =\int_{q_0-2t}^{q_0} F(z)\, dz < 2pt \quad \text{and} \quad  \G_t(q_0+t) =\int_{q_0}^{q_0+2t} F(z)\, dz  \geq 2pt
\end{equation*}  
 then, from the continuity of $\G_t(\cdot)$, it follows that for all fixed $t\ge 0$, there exists an $x \in [q_0-t, q_0 +t] $ such that $  \G_t(x)= 2pt$.
 On the other hand, for all $x\in (q_0-t,q_0+t)$, $x+t > q_0$ and then  $F((x+t)^\pm)\ge p$ whereas $x-t<q_0$ and then $F((x-t)^\pm)<p$. As a consequence,
  for  $x\in [q_0-t,q_0+t)$, 
  \begin{equation}
\partial_\pm \G_t(x)=F((x+t)^\pm)-F((x-t)^\pm)>0
  \end{equation}
   i.e. $\G_t(\cdot)$ is strictly increasing in this interval, and then invertible as a function from $[q_0-t,q_0+t]$ to its image.  This proves the lemma.
\end{proof}

\begin{theorem}[Existence and uniqueness]\label{TeoEx}
 The system  \eqref{wt}-\eqref{qb0} has a unique solution $( q_t)_{t\ge0}$ given by
 \begin{align}
   \label{22}
   q_t =
   \begin{cases}
    \G_t^{-1}(2pt)&  \G_t^{-1}(2pt)\in(0,1)\\
     1,&\G_t^{-1}(2pt)>1\\
      0,&\G_t^{-1}(2pt)<0
   \end{cases}.
 \end{align}
The bid distribution at time $t$, denoted $\mu\cT_t$ is given by
 \begin{align}
   \label{11}
  d( \mu\cT_t)(x) = \delta_{q_t}(x)\,\mu([q_t-t,q_t+t)) + d\mu(x-t)\one\{x<q_t\}+ d\mu(x+t)\one\{x>q_t\}.
 \end{align}
\end{theorem}
\begin{proof}
Let $(q_t)_{t\ge 0}$ be the solution of the system  \eqref{wt}-\eqref{qb0}. Then, as long as $q_t\notin \{0,1\}$, it satisfies:
   \begin{equation}
    \label{eq:25_micro}
     \dot{q_t}  = \frac{2p - F(q_t+t)- F((q_t-t)^-)}{ F(q_t+t)- F((q_t-t)^-)} 
  \end{equation}
  then
  \begin{align}\label{qui}
    2p&=( F(q_t+t)- F((q_t-t)^-))\,  \dot{q_t} + F(q_t+t) +F((q_t-t)^-)\\
       &=   \partial_+ G(q_t+t)( \dot{q_t}+1) - \partial_-G(q_t-t)( \dot{q_t}-1) \nonumber
  \end{align}
  On the other hand we have that
  \begin{align*}
  &  \frac {d} {dt}\G_{t}(q_t):=\lim_{\eps \to 0} \frac{  \G_{t+\eps}(q_{t+\eps})-  \G_t(q_t)}{\eps}\\
   &= \lim_{\eps \to 0} \frac{  G(q_{t+\eps}+t+\eps)-  G(q_t+t)}{\eps} -   \lim_{\eps \to 0} \frac{  G(q_{t+\eps}-t-\eps)-  G(q_t-t)}{\eps} \\
  \end{align*}
  Let $\kappa_\eps:= \eps+q_{t+\eps}-q_t$, then $\lim_{\eps \searrow 0} \frac {\kappa_\eps}\eps=1+\dot{q}_t$, and,  since $\dot q_t \ge -1$, for any $\eps \in \mathbb R$ with $|\eps|$ small enough, $\kappa_\eps \ge 0$, then
   \begin{align*}
\lim_{\eps \to 0} \frac{  G(q_{t+\eps}+t+\eps)-  G(q_t+t)}{\eps}&= \lim_{\eps \to 0} \frac{  G(q_t+t+\kappa_\eps)-  G(q_t+t)}{\kappa_\eps}\cdot \frac{\kappa_\eps}\eps \\
&= \partial_+G(q_t+t)(1+\dot{q}_t).
     \end{align*}
     Analogously,  let $h_\eps:= \eps- (q_{t+\eps}-q_t)$, then $\lim_{\eps \to 0} \frac {h_\eps}\eps=1-\dot{q}_t$, and,  since $\dot q_t\le 1$, for any $\eps \in \mathbb R$ with $|\eps|$ small enough, $h_\eps \ge 0$, then
   \begin{align*}
 \lim_{\eps \to 0}  \frac{  G(q_{t+\eps}-t-\eps)-  G(q_t-t)}{\eps} &= \lim_{\eps \to 0} \frac{  G(q_t-t-h_\eps)-  G(q_t-t)}{h_\eps}\cdot \frac {h_\eps}{\eps} \\
 &= -\partial_-G(q_t-t)(1-\dot{q}_t)
  \end{align*}
  and then
    \begin{align*}
    \frac {d} {dt}\G_{t}(q_t)&= \partial_+G(q_t+t)(\dot{q}_t+1)+\partial_-G(q_t-t)(\dot{q}_t-1)
  \end{align*}
  Then, using \eqref{qui} we conclude that
  \begin{equation}
  2p=\frac {d} {dt}\G_{t}(q_t) \quad\text{and}\quad 2pt=\G_{t}(q_t).
  \end{equation}
 On the other hand, since  $q_t$ has  velocity $u_t \in [-1,1]$ we deduce that  $ q_t \in [q_0 - t, q_0 + t]$. From these facts we conclude that, as long as  $q_t\notin \{0,1\}$, it coincides with $x_t:\G_t^{-1}(2pt)$, i.e. the unique solution of the problem \eqref{35}.

Finally, \eqref{11} follows from \eqref{17}.
\end{proof}

\paragraph{Barycenter decomposition.}

We provide a physical interpretation of \eqref{35}. 
Let $0\le \alpha\le \beta\le 1$, we denote by $\bari(\alpha,\beta]$ the barycenter of the mass  restricted to the interval  $(\alpha,\beta]$:
\begin{align}
  \bari[\alpha,\beta]&:= \frac{1}{F(\beta)-F(\alpha^-)}\int_{(\alpha,\beta]} z \,d\mu(z) 
= \beta-\int_{(\alpha,\beta]}\frac{F(z)-F(\alpha^-)}{F(\beta)-F(\alpha^-)}\, dz \label{bc1}
\end{align}
\begin{proposition}
The  solution $(q_t)_{t\ge 0}$ of the equation
\begin{align}
  \label{34}
  q_t = \bari[q_t-t,q_t+t] + t\;\frac{p-F((q_t-t)^-)-(F(q_t+t)-p)}{F(q_t+t)-F((q_t-t)^-)}
\end{align}
is the unique solution of the system  \eqref{wt}-\eqref{qb0}.
\end{proposition}

\begin{proof}
We prove that  a function $(x_t)_{t\ge 0}$, is  solution of the problem \eqref{35}
if and only if it satisfies \eqref{34}.
Formula \eqref{bc1} is equivalent to
\begin{equation}\label{bc1*}
\int_\alpha^\beta F(z)\, dz=(\beta-\bari[\alpha,\beta])F(\beta)+(\bari[\alpha,\beta]-\alpha)F(\alpha^-).
\end{equation}
Using this and the definition of $\G_t$ we get
\begin{align}\label{29}
  \G_t(x) &= \bigl(x+t- \bari[x-t,x+t]\bigr)\, F(x+t)+ \bigl(\bari[x-t,x+t]-x+t\bigr)\,  F((x-t)^-) \nonumber \\
             &= \bigl(x- \bari[x-t,x+t]\bigr)\, \bigl(F(x+t)-F((x-t)^-)\bigr)
               +t (F(x+t) +F((x-t)^-)).
\end{align}
Then, equation \eqref{34} follows by imposing $\G_t(x)=2pt$.
\end{proof}

\begin{remark}
  \label{16} \rm
  The quantity $\bari[q_t-t,q_t+t]$ is the barycenter at time $0$ of the initial bids that coincide with the max-price at time $t$. The bids hitting the max-price before time $t$ initially below $p$ have total mass $p-F((q_t-t)^-)$ and move at speed 1, those above $p$ have total mass $F(q_t+t)-p$ and move at speed $-1$; the denominator of the second summand of \eqref{34} is the sum of those masses. Since the total moment is conserved when the bids attain the max-price, the barycenter of this mass at time $t$ is given by the right hand side of \eqref{34}. Since at time $t$ this mass is concentrated in $q_t$, we get the identity \eqref{34}. This identity is equivalent to
    \begin{align}
        q_t = \bari[q_t-t,q_t+t] + t\;\frac{w_t-\ell_t}{w_t+\ell_t}, \quad t\ge0,
    \end{align}
\end{remark}

\begin{example}[Solution for the case $F(x)=x$]  \rm When $\mu$ is the Lebesgue measure on $[0,1]$, the solution $q_t$ of Theorem \ref{TeoEx} can be explicitly computed.
Assume $p \leq \frac12$ and $F(x)=x$ in the interval $[0,1]$. Then, the solution is
  \begin{align}
    \label{eq:5}
    \text{If $p \leq \frac14$, then }
    \quad  q_t &=
          \begin{cases}
            p,& 0\le t\le p
            \\
            -t + 2\sqrt{pt}, & p\le t\le 4p\\
            0,& t\ge 4p
          \end{cases}.\\
    \text{If $\frac14<p<\frac12$, then }
    \quad  q_t &=
   \begin{cases} \label{eq26}
            p,& 0\le t\le p
            \\
     -t + 2\sqrt{pt}, & p\le t\le \frac1{4p}\\
     (2p-1)t + \dfrac{1}{2},  &\frac1{4p} \leq t \leq \frac{1}{2(1-2p)}\\
     0,& t \geq \frac{1}{2(1-2p)}
          \end{cases}.       \\
           \text{If $p=\frac12$, then }
    \quad  q_t &=        \frac12          \text{ for all }  t \geq 0     \;.          
\end{align}
When $x \in (0,1)$ one has that
\begin{align}
\G_t (x)&=
   \begin{cases}
            2xt, & t \leq x \leq 1-t
            \\
      \frac{(x+t)^2}{2}, &  x\leq t, \ x \leq 1-t
      \\
      x+t- \frac{1}{2}, & x\leq t, \ x \geq 1-t
          \end{cases}     . 
\end{align}
The solution $q_t = \G^{-1}_t(2p)$ matches Proposition 8 of \cite{psk24} and it is represented in Figure \ref{graficop03} below.
\end{example}

\begin{figure}[th]
  \centering
  \includegraphics[width= .7\textwidth]{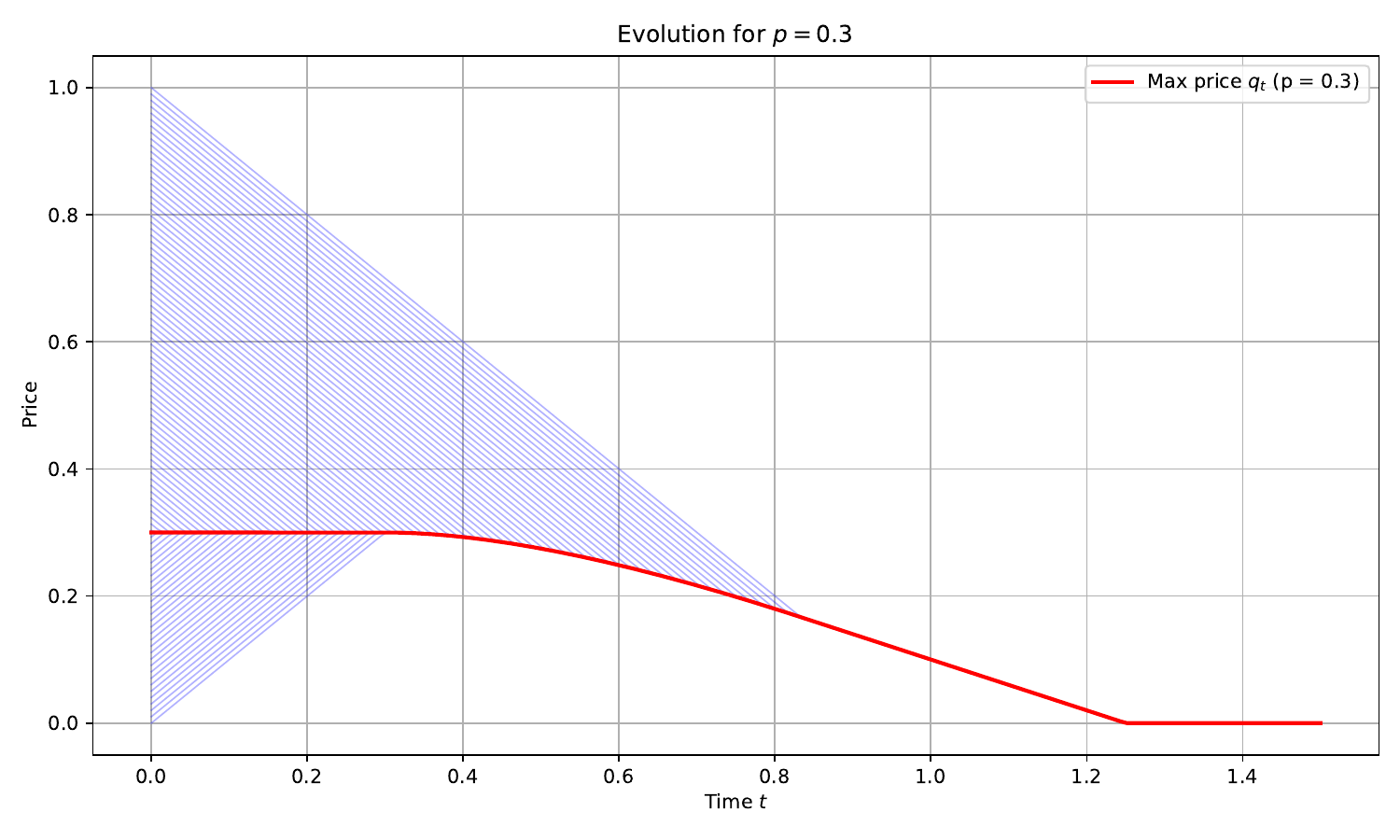}
  \caption{Max-price function $q_t$ computed in equation \eqref{eq26} when $p=0.3$ the initial measure  bids distribution is uniform in $[0,1]$. }
  \label{graficop03}
\end{figure}

\subsection{Initial delta measures}
\label{s21}

We consider an initial combination of delta measures and provide an alternative construction of the solution of  the system  \eqref{wt}-\eqref{qb0}. Let $\mu = \frac1n\sum_{b\in B}\delta_b$, for a multi-set $B$ with $n$ elements in $[0,1]$ and $F$ its cumulative distribution. Let $q_t$ be the evolution of the max-price.  The evolution $b_t=b_t[b,B]$ of each initial bid $b\in B$ is given by 
\begin{align}
 b_t  &=
   \begin{cases}
     b+t &  b<q_0, \;b+t<q_t\\
     q_t&  b\ge q_0, \;b-t\le q_t\text{ or } b\le q_0, \;b+t \ge q_t \\
     b-t&  b>q_0, \;b-t>q_t
   \end{cases}
\end{align}
Since $|u_t|\le 1$ for all $t$, we have $q_t-t\le q_0 \le q_t+t$, and \eqref{17} gives
\begin{align}
  \label{bt37}
 b_t  &=
   \begin{cases}
     b+t & \text{for }  \; b<q_t-t\\
     q_t&  \text{for }  \; q_t-t \le  b \le q_t+t \\
     b-t&  \text{for }  \; b-t>q_t +t\, .
   \end{cases}
\end{align}
Denote $\cT_tB$ the multi-set (we use square brackets to denote multi-sets), 
\begin{align}
  \label{13}
  \cT_tB:= \bigl[b_t[b,B]: \, b\in B\bigr].
\end{align}
The operator $\cT_t$ describes the evolution of the bids satisfying  the system  \eqref{wt}-\eqref{qb0}. The corresponding evolution of the measure $\mu$ is
\begin{gather}
  \label{eq:36}
  \mu \cT_t:= (w_t+\ell_t)\,\delta_{q_t}  + \frac1n\sum_{b\in B} \one\{b_t\neq q_t\}\, \delta_{b_t},\notag\\
w_t+\ell_t=  F(q_t + t) - F((q_t - t)^-)= \frac1n\sum_{b\in B} \one\{ q_t-t< b \le q_t+t\}.\notag
\end{gather}
The dynamics of a bid $b_t$ is ballistic until hitting the max-price trajectory $q_t$. If $b_{s}\neq q_{s}$ for $s<t$ and $b_t=q_t$, meaning that it hits the max-price at time $t$, from that moment on, the 
bidder travels with max-price, whose velocity $u_t$ is updated according to the new updated values of $w_t$ and $\ell_t$. The max-price velocity $u_t$, $t\ge 0$, is piecewise constant with discontinuities at collision times, at which it is right-continuous. 
The max-price function $q_t$, $t\ge 0$, is then continuous and piecewise linear.

The above construction gives an alternative way to show the existence of a unique solution for the system  \eqref{wt}-\eqref{qb0} with bid distribution given by a finite number of delta functions, as follows. 

\begin{theorem}[Evolution of sums of delta measures]
  For $\mu=\frac1n\sum_{b\in B}\delta_b$, for $B\subset[0,1]$ with $n$ bids  the system  \eqref{wt}-\eqref{qb0} has a unique solution $(B_t,q_t)$, $t\ge 0$, which is ballistic except at the collision times.
\end{theorem}
\begin{proof} Since there is  a finite number of bids at time 0, there is also a finite number of collisions between $q_t$ and $b_t$'s. At collision times the speed $u_t$ of $q_t$ is well defined by \eqref{ut}. Whereas the velocity   of the bids $b_t$ that have not yet reached  the max-price is equal to $1$ or $-1$ according to  \eqref{bt37}. Between two consecutive collisions the dynamics is just ballistic. This shows that the system has a unique solution. 
\end{proof}

\section{Discrete model}\label{discretemodel}
In this section we study the PSK model \cite{psk24} for fixed $\gamma$, using the continuous approach of the previous section. We obtain an explicit formula for the max-price evolution with an error of at most $3\gamma$. Then we show that when the initial bids are concentrated in 3 consecutive sites  $|b_i-b_j|\le 3\gamma$ for all $i,j$, the motion of the max-price is periodic. 
The evolution of the bids and the max-price is periodic and the period depends on whether the number of winners $w$ and the number of losers $\ell$ are coprime.


\subsection{Max-price evolution}
Let $n\in \NN$ and let $\gamma$ be the inverse of a natural number called the learning parameter. We consider discrete bids belonging to $I^\gamma:=[0,1] \cap \gamma \mathbb N$ and a multi-set $B$ with elements  $b\in I^\gamma$.
Each bid is associated with a mass $1/n$, so that the total mass of a bid configuration is $1$.
Fix a parameter $p\in [0,1]$ such that $np\in \NN$. Let $b'_1(B),\dots,b'_n(B)$ be the ordered elements of the multi-set $B$ and denote by $q(B):=b'_{np}$ the $p$-quantile of $B$. Assume that at each time-step $\gamma$, the bids $b_1, \ldots, b_{np}$ win the auction and the remaining bids lose it, so that $q(B)$ is the max-price paid by the buyer. 
For the next round, each winner increases her bid by $\gamma$ and each loser decreases hers by $\gamma$. After the first round, the new bids are denoted by  $\rmT^{\gamma}B$, defined by
\begin{align}
  \rmT^{\gamma}B &:= [b'_i(B)+\gamma \bigl(\one\{0<i\le np\} - \one\{np<i\le n\}\bigr):  i=1, \ldots, n].\label{tb11}
\end{align}
Iterating  $\lfloor t/\gamma\rfloor$ times the operator $\rmT^{\gamma}$ we get the bid evolution and the max-price at time $t$, 
  \begin{align}
  \rmT^{\gamma}_tB &:=  (\rmT^{\gamma})^{\lfloor t/\gamma\rfloor}B, \qquad 
  q^\gamma_t := q(\rmT^{\gamma}_tB)= b_{np}'(\rmT^{\gamma}_tB). \label{23}
\end{align}
Denote the cumulative distribution, its integral, and  the barycenters associated to $B$ by
\begin{gather}
F(x):=\frac{1}{n} \sum_{i=1}^n \one\{b_i\le x\},\qquad
 \G_{t}(x):=\int_{x-t}^{x+t} F(z) dz,\\[2mm]
   \label{ga9}
    \bari:=\frac{\sum_{i=1}^n b_i}{n},\qquad\qquad
   \bari[\alpha, \beta]:=\frac{\sum_{i=1}^n b_i \one_{b_i \in (\alpha, \beta]} }{\sum_{i=1}^n \one_{b_i \in (\alpha, \beta]} }.
  \end{gather}

 \begin{proposition}[Max-price evolution] At time $t \ge 0$, the max-price evolution $q^\gamma_t$  satisfies the following inequalities
\begin{gather}
  \label{34d}
  \Biggl| q^\gamma_t - \bari[q^\gamma_t-t,q^\gamma_t+t] -
  \frac
  {2p-F((q^\gamma_t-t)^-)-F(q^\gamma_t+t)}
  {F(q^\gamma_t+t)-F((q^\gamma_t-t)^-)} \;  \cdot t
  \Biggr|\le\; 2\gamma,\\[2mm]
  \bigl|\G_t(q^\gamma_t)-2pt\bigr| \le 2\gamma. \label{34d*}
\end{gather}
\end{proposition}

\begin{proof}
Consider the bids that started in the interval $[q^\gamma_t-t,q^\gamma_t+t]$ at time $t$. Among them, there are  $n(p-F((q^\gamma_t-t)^-)))$ winners who increment their bid by $\gamma$, and $n(F(q^\gamma_t+t)-p)$ losers who decrease their bid by $\gamma$. Hence, their barycenter at time $t$ is
\begin{equation*}
\bari[q^\gamma_t-t,q^\gamma_t+t] +t\;
  \frac
  {2p-F((q^\gamma_t-t)^-))-F(q^\gamma_t+t)}
  {F(q^\gamma_t+t)-F((q^\gamma_t-t)^-))}.
  \end{equation*}
 Once a bid's distance from the max-price falls below $2 \gamma$ it will remain within $2\gamma$ of the max-price thereafter. As a consequence, at time $t$ all the bidders we are considering are in the interval $[q^\gamma_t-2\gamma,q^\gamma_t+2\gamma]$ and therefore also their barycenter must be at a distance of at most $2\gamma$ from the max-price $q^\gamma_t$.  This concludes the first part of the proof.
To prove the second inequality we notice that, as in \eqref{29}, 
\begin{align}\label{291*}
  \G_t(x) &= \bigl(x- \bari[x-t,x+t]\bigr)\, \bigl(F(x+t)-F((x-t)^-)\bigr) \nonumber \\
             & \hskip3cm+t (F(x+t) +F((x-t)^-)),
\end{align}
 and to get \eqref{34d*} multiply both sides of \eqref{34d} by $F(q^\gamma_t+t)-F((q^\gamma_t-t)^-)$.
\end{proof}

\subsection{Periodicity of the max-price evolution}

We prove now that the max-price in the bulk is periodic. Denote $[b_1,\dots,b_n]+a:= [b_1+a,\dots,b_n+a]$. Moreover for two elements $b,b'$ of a multiset $B$ we define the relation 
$b\sim b'$ if and only if $\frac{|b-b'|}{\gamma}$ is even. Then $\sim$ defines an equivalence relation on $B$ that partitions it into two equivalence classes $\hat B$ and $\check B$, with $\hat B  \cup \check B=B$.

\begin{lemma}[Bulk periodicity of the max-price]\label{pe4}
  Let $B$ be a multiset with all bids belonging to three consecutive sites, that is,  $|b-\tb|\le 2\gamma$ for all $b,\tb\in B$. Let $\hat B, \check B \subseteq B$ as above. Let $p \in [0,1]$ such that $w=pn \in \mathbb{N}$ is the total number of winners and  $\ell=(1-p)n \in \mathbb{N}$ is the total number of losers at each step.

  a.  Bulk periodicity. Assume that for all $k\in\{0,\dots,w+\ell\}$, $(\rmT^{\gamma})^kB$ does not interact with the boundary, meaning $(\rmT^{\gamma})^kB \cap\{0,1\}=\emptyset$, then
\begin{align}
(\rmT^{\gamma})^{K}B& = B+\gamma \cdot \frac{w-\ell}{w+\ell} \cdot K  \;.
\end{align}
where $K=\frac{w + \ell}{gcd(w, \ell)}$ if $|\hat B|\neq |\check B|$, and
\begin{equation}
K=\left\{ 
\begin{array}{ll}
\frac{w + \ell}{gcd(w, \ell)} & \text{if } \frac{w + \ell}{gcd(w, \ell)} \text{ is odd}\\
\\
\frac{w + \ell}{2\cdot gcd(w, \ell)} & \text{if } \frac{w + \ell}{gcd(w, \ell)} \text{ is even}
\end{array}
\right. \qquad \text{if} \qquad |\hat B|= |\check B|.
\end{equation}

b. Boundary behavior. For $w < \ell$ the configuration $B= [0, \ldots , 0, \gamma, \ldots, \gamma]$ with the first $\ell$ bids in position $0$ and the remaining $w$ bids in position $\gamma$ is conserved by $ \rmT^{\gamma}$. Similarly, for $w> \ell $, the configuration $B= [1-\gamma, \ldots , 1-\gamma, 1, \ldots, 1]$ with the first $\ell$ bids in position $1 - \gamma$ and the  remaining $w$ bids in position $1$ is conserved by $ \rmT^{\gamma}$; see also Proposition 2 of \cite{psk24} where these two configurations are referred to as locally asymptotically stable. 
\end{lemma}
\begin{figure}[th]
  \centering
  \includegraphics[width= .7\textwidth]{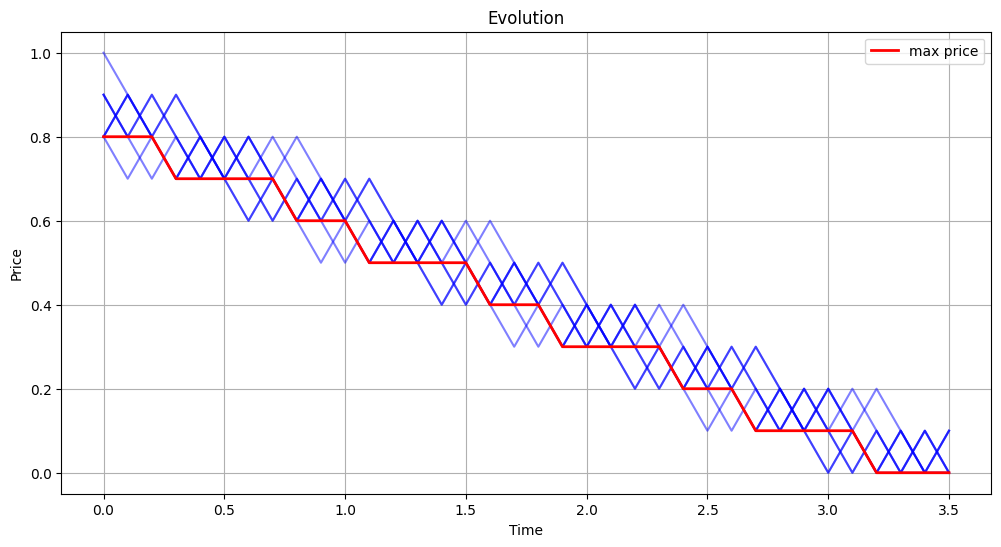}
  \caption{Bulk periodicity. The increments of the max-price are periodic. In the picture there are $3$ winners and $5$ losers with $\gamma = 0.1$. The period is $8$. Superpositions are indicated with darker blue. }
  \label{per}
\end{figure}

\begin{proof}
a. Let $t$ be a multiple of $\gamma$ with $t<\gamma(w+\ell)$ and assume that the bids at time $t$ are at sites $z-\gamma, z, z+\gamma$. If we have  $q^\gamma_{t}=z$, then at time $t + \gamma$ the bids are again at sites $z-\gamma, z, z+\gamma$. If instead $q_t=z\pm\gamma$ then at time $t+\gamma$ the bidders are at sites $z, z\pm\gamma, z \pm2\gamma$.  As a consequence we have that, given our starting configuration, all the bidders will always be on at most three consecutive sites, therefore since the dynamics is deterministic and the number of admissible relative configurations is finite, every trajectory eventually becomes periodic.

In order to simplify the notation, we denote the relative configuration of the particles by the triplet $(n_1,n_2,n_3)$ where $n_1$ particles are bidding a value $i \in I^\gamma$, $n_2$ particles are bidding a value $i+\gamma \in I^\gamma$ and $n_3$ particles are bidding a value $i+2\gamma \in I^\gamma$. 

Observe that, at each time step, every particle moves either upward or downward by $\gamma$. The distance between two particles after each time step can then either stay the same or change by $2\gamma$. Consequently, two particles that in the initial configuration are in the same class, namely both are in $\hat B$ or both are in $\check B$, will still be in the same class in the dynamics.  Let $x$ and $y$ denote the cardinality of $\hat B$ and $\check B$. Since the members of one of the classes will always be concentrated in the central site of the three-site configuration, every admissible configuration satisfies either $n_2 = x$ or $n_2 = y$ and $x+y=w+\ell = n$.

If $w=\ell$ it is trivial to check that $K=2$ if $x\neq y$ and $K=1$ if $x=y$.

Consider the case $w<\ell$; the proof of the case $\ell > w$ is entirely analogous.
If the system is in a configuration with $n_1\ge w$ then at the next step it will be in the configuration $(n_1-w,n_2,n_3+w)$, and we will call this a \textit{descent step}. All the configurations obtained from a descent step have the value $n_1$ in the same class modulo $w$ by definition. 

Let $GCD$ be the greatest common divisor between $w$ and $\ell$; $GCD := gcd(w, \ell)$.
First, we will consider the simpler case when $w$ and $\ell$ are coprime, namely $GCD = 1$ and $x < y$.
In this case, we can show that, starting from any of the possible triplets, the system will cycle through all of the triplets of the form $(n,x,y-n)$ with $0\le n<y$ and all the triplets of the form $(m,y,x-m)$ with $0\le m<x$ before it gets back to the starting one. Therefore the length of the cycle is $x+y = w + \ell$.

Without loss of generality we start with a configuration with $n_2 = x$, i.e. $(n,x,y-n)$ with $0 \leq n < y$.
The key idea is to show that, before returning to the initial configuration, we have visited all the others for $ n_1 \in \lbrace 0, 1, \ldots, y-1 \rbrace$ exactly once.

The starting configuration is $(n,x,y-n)$ and after a finite number of descent steps (possibly zero) the system will eventually reach the configuration $(a,x,y-a)$ with $0\le a<w$ and $a \equiv n \pmod{w}$.

\begin{itemize}
\item[i] \textbf{Case $x \geq w$: } First we focus on the configurations with $n_2 = x$ and by evolving the system from the initial configuration we want to count the configurations of the type $(n_1, x, y-n_1)$.
The configuration at  the next step is
$(x-w+a,y,w-a)$ and after a finite number of descent steps the system will eventually reach the configuration $(b, y, x-b)$ with $0\le b<w$ and $b \equiv x-w+a \pmod{w}$. At the next step the system will then be in the configuration $(y-w+b,x,w-b)$ and we notice that $n_1$ is $y-w+b \equiv y+b  \equiv y+x+a \equiv a+\ell \pmod{w}$ and that $y-w+b$ is the highest number less than $y$ in the class $[a+\ell ] \pmod{w} $.
We continue with descent steps, such that $n_1$ takes all the possible decreasing values between $[0,y)$ of the class $[\ell + a] \pmod{w} $ until it takes  the value of the canonical representative of the class, namely the one in $[0,w)$.
This mechanism repeats analogously until the next time $n_2 = x$, the term $n_1$ of the configuration is equivalent to $a+2\ell \pmod{w} $ and so on until the term $n_1$ is equivalent to $a+(w-1) \ell \pmod{w}$. All the classes $[a], [a + \ell], \ldots [a+(w-1) \ell]$ are distinct because $\ell$ and $w$ are coprime and the following class visited by the system is $a+w \ell \equiv a  \pmod{w}$, which is the starting class. Once in this class, after a finite number of steps, the system will visit the starting configuration $(n, x, y-n)$: the cycle is closed and all configurations with $n_2=x$ with $n_1 \in \lbrace 0, 1, \ldots, y-1 \rbrace$ have been visited. 

With a similar reasoning all the configurations with  $n_2 = y$ and $n_1 \in \lbrace 0, 1, \ldots x-1 \rbrace$ have been visited as we wanted to show. 

\item [ii]  \textbf{Case $x < w$: } Notice that every time the system is in the configuration $(c, x, y-c)$ with $0 \leq c < w $ we need to distinguish two different subcases: 

\begin{itemize}
\item  {If $  n_1 + n_2 \geq w$: }  the next configuration  have $n_2 = y$ and as in case [i]  the following time $n_2 = x$, $n_1$ is the highest number of the class $[\ell + c]$ less than $y$.

\item   {If $ n_1 + n_2 < w$: }  the next configuration has $n_2 = x$, namely $(\ell +c, x, w -x -c)$
where $n_1= y+x+c-w = \ell + c$ is the highest number less than $y$ in the new class $[\ell + c]$.
\end{itemize}
In both cases all the classes $[a+ k*\ell]$ are visited for $k= 1,2, \ldots, (w-1)$. Therefore when $n_2 = x $ all the possible values  $n_1 \in \lbrace 0, 1, \ldots y-1 \rbrace$ are taken. It is left to show that also when 
$n_2 = y $ all the possible values  $n_1 \in  \lbrace 0, 1, \ldots x-1 \rbrace $ are taken. 

This must be the case since for any $0 \leq n_1 < x $  fix $a^* = n_1 - x + w$ and consider the configuration $(a^*, x, y - a^* )$, observe that $a^*  < w$ and $a^* + x \geq w$, therefore the next configuration is $(x-w+a^*, y, w-a^*)$ which, by definition of $a^* $ is $(b, y,x-b)$ and this concludes the proof. 
\end{itemize}

If $x=y$, necessarily $x \geq w$ and the same reasoning as in case [i] applies. 
The two families of configurations $(n_1,x,x-n_1)$
obtained by fixing either color class coincide. Hence each relative
configuration is represented twice in the previous counting argument and therefore the period is halved: $\frac{w + \ell}{2}$.

Now we generalize to the case when $w$ and $\ell$ are not coprime, $w = w' \cdot GCD$ and $\ell = \ell' \cdot GCD$ and $x<y$.
As before we start with the general configuration $(n,x,y-n)$ and we focus on the configuration with $n_2=x$. The same mechanisms of items [i] and [ii] apply, however we notice that the distinct classes visited by the system are  $[a], [a+ \ell], \ldots [(w'-1) \ell +a ]$ since $a \equiv w'\ell + a \pmod{w}$. Since $\ell = \ell' \cdot GCD$ all these classes are the same $\pmod{GCD}$ and as a consequence, all the configurations  $(n_1, x , y  - n_1)$ with $0 \leq n_1 < y$ and $n_1 \equiv n \pmod{GCD}$ are visited.


For configurations of the form $(n_1,y,x-n_1)$, the residue class of
$n_1$ modulo $\mathrm{GCD}$ is also preserved. Indeed, such a
configuration can arise in two ways.

If it is obtained through a descent step, then only $w$ particles are
moved from the leftmost site to the rightmost one, so the value of
$n_1$ changes by a multiple of $w$. Since $\mathrm{GCD}\mid w$, the
residue class of $n_1$ modulo $\mathrm{GCD}$ remains unchanged.

Otherwise, it is obtained from a configuration $(c,x,y-c)$ with
$c+x\geq w$. In this case the next configuration is
$
(x-w+c,\; y,\; w-c),
$
so that
$
n_1=x-w+c.
$
Since both $w$ and $x$ are multiples of $\mathrm{GCD}$ modulo
$\mathrm{GCD}$, we obtain
\[
n_1 \equiv x+c \equiv x+n \pmod{\mathrm{GCD}},
\]
where $c\equiv n \pmod{\mathrm{GCD}}$.
Hence every configuration reached from the initial state belongs to a
fixed residue class modulo $\mathrm{GCD}$.

Moreover, all these configurations are visited fby the same reasoning as in cases [i] and [ii].

We now partition the set of all possible configurations with $n_2 = x$ or $n_2 = y$ fixed into $GCD$ sets and we show that each of these sets has the same cardinality. 
For configurations of the type $(\tilde{n}_1, x, y-\tilde{n}_1)$ we know that $0 \leq \tilde{n}_1<y $ and we can always write $y = \tilde{k}\cdot GCD + \tilde{r}$ with $0 \leq \tilde{r} < GCD$ and for $0 \leq \tilde{n}_1 < \tilde{k} \cdot GCD $, each class has the same number of elements. 
For configuration of the type $(\tilde{\tilde{n}}_1, y, x- \tilde{\tilde{n}}_1)$ we know that $0 \leq \tilde{\tilde{n}}_1 <x $ and we can always write $x = \tilde{\tilde{k}} \cdot GCD + \tilde{\tilde{r}}$ with $0 \leq \tilde{\tilde{r}}< GCD$ and for $0 \leq \tilde{\tilde{n}}_1< \tilde{\tilde{k}} \cdot GCD $, each class has the same number of elements. 
Since $x+y$ is a multiple of $GCD$, either $\tilde{r} + \tilde{\tilde{r}}  = 0 $, in which case we are done, or $\tilde{r} + \tilde{\tilde{r}}  = GCD $. To show that each class gains exactly one element
we notice that $\tilde{k}\cdot GCD  \leq \tilde{n}_1 < \tilde{k} \cdot GCD  + \tilde{r} $ and the $\tilde{r}$ classes are distinct. Similarly $\tilde{\tilde{k}} \cdot GCD  \leq \tilde{\tilde{n}}_1 < \tilde{\tilde{k}} \cdot GCD  + \tilde{\tilde{r}} $ and the $\tilde{\tilde{r}}$ classes are distinct. They are also globally distinct since 
\begin{equation*}
\tilde{\tilde{n}}_1 \equiv \tilde{n}_1 + x =  \tilde{n}_1 + \tilde{\tilde{k}} \cdot GCD +  \tilde{\tilde{r}}  \equiv  \tilde{n}_1 + \tilde{\tilde{r}}  \pmod{GCD} \quad \text{and} \quad  n_1 \in [0, GCD) \pmod{GCD}.
\end{equation*}
Since we have $w + \ell$ total configurations, each cycle has length $\frac{w + \ell}{GCD}$. 

When $x=y$, notice that if in the initial configuration $n_1 \equiv n \mod GCD$ then it will always be in the classes $[n]$ and $[x+n] \mod GCD$. Observe that if $GCD | x$ then those are the same class and therefore proceeding as before we would again count each configuration twice and we have to halve the period. As we can write $x=GCD \cdot \frac{w+\ell}{2 \cdot GCD}$ this happens if $\frac{w+\ell}{GCD}$ is even.

To conclude the proof notice that since at each time step $w$ particles go up by $\gamma$ and $\ell$ move down by $\gamma$, at each time step the barycenter moves by $\gamma \cdot \frac{w-\ell}{w+\ell}$. Since after $K$ time steps the relative configuration of the particles will repeat, also their distances from the barycenter will be the same, which means each of them moves by $\gamma \cdot \frac{w-\ell}{w+\ell} \cdot K$.

b. We consider again the case $w< \ell$. It is clear from the definition of the dynamics \eqref{23} that once the system reaches the configuration with $\ell$ bids at the first site, zero and $w$ bids on the second site $\gamma$, at all the next steps, $w$ of the $\ell$ bids at site zero will jump to site $\gamma$ and the remaining $\ell - w$ will stay at site zero together with the $w$ losers from $\gamma$. Moreover, notice that the system will eventually reach the above configuration because the barycenter moves always down until there are $\ell$ bids at site zero and it is forbidden by the dynamics to have more than $\ell$ bids in zero. Once there are $\ell$ bids in zero, either the remaining $w$ bids are in site $\gamma$ and we are done or at the next step there are less than $\ell$ bids in zero and the barycenter keeps moving down. 
\end{proof}

\section{The discrete model converges to the continuous one}
\label{s4}
We now compare the discrete PSK dynamics of Section \ref{discretemodel}  with the continuous evolution constructed in Section \ref{Cont}. 
The discrete model is rescaled in time and space with mesh of size $\gamma$ and we show that both the discrete max-price and the individual bid trajectories converge to their continuous counterparts as $\gamma$ goes to zero. 

To this end, let $B$ be a multi-set with $n$ elements of $[0,1]$, $B=[b_1, \ldots, b_n]$ with  $b_1\le \ldots\le b_n$.  
Let $B^{\gamma}$, the $\gamma$-discretization of $B$, be the multi-set with elements in $I^\gamma:=[0,1] \cap \gamma \mathbb N$, given by
 \begin{align}
    \label{18}
 B^{\gamma} := [ b^\gamma_1,\dots,b^\gamma_n],\qquad b^\gamma_i:= \gamma \lfloor b_i/\gamma\rfloor.
  \end{align}
  For time $t\ge 0$ define the operator $\rmT^{\gamma}_t$ and the max-price at time $t$ by
\begin{align}
    \rmT^{\gamma}_t B &:=(\rmT^{\gamma})^{\lfloor t/\gamma\rfloor} B,\quad
 q^\gamma_t:= b'_{pn}\bigl(\rmT^{\gamma}_t B\bigr). \label{tb3}
\end{align}
  \begin{theorem}\label{discr-cont}
    Let $p\in [0,1]$ with $pn \in \mathbb N$. Let $B$ be a multi-set with elements in $[0,1]$ and $B^\gamma$ its $\gamma$-discretization with elements in $I^\gamma$. Let $\cT_tB$ be the continuous evolution of $B$ in \eqref{13} and $q_t$ the corresponding max-price. Moreover,  let $\rmT^{\gamma}_tB^\gamma$ be the discrete evolution of $B^\gamma$ in \eqref{tb3} and $q^\gamma_t$ its max-price. Then
  \begin{align}
    \label{20}
    \sup_{t\ge 0}\bigl|  q_t^\gamma - q_t\bigr|&\le (3n+2)\gamma,\\
    \sup_{b\in B}\sup_{t\ge 0}\bigl|  b_t^\gamma - b_t\bigr|&\le (3n+2)\gamma.\label{19}
 \end{align}
\end{theorem}

\begin{proof}
  The bound \eqref{19} is a consequence of \eqref{20}. 
To show \eqref{20}, we use the compact notation:
\begin{eqnarray}
&F(x):=F_{B}(x), &  \qquad \G(x):=\G_t[B](x), \\
&F^\gamma(x):=F_{B^{\gamma}}(x),& \qquad \G^\gamma_t(x):=\G_t[B^{\gamma}](x).
\end{eqnarray}
Set $\tau=\lfloor t/\gamma\rfloor\gamma$ and use \eqref{34d*} to get
\begin{equation}
\bigl|2p\tau-\G_\tau^\gamma(q^\gamma_\tau)\bigr| \le 2\gamma.\label{15}
\end{equation}
Observe that
\begin{align}
\bigl|\G_\tau(q^\gamma_\tau)-\G^\gamma_\tau(q^\gamma_\tau)\bigr| &\le 
\int_{q_\tau^\gamma-\tau}^{q_\tau^\gamma+\tau} \bigl| F^\gamma(x)-F(x) \bigr| dx
\le
\int_{0}^{1} \bigl| F^\gamma(x)-F(x) \bigr| dx \\
&\le 
\sum_{k=1}^\frac{1}{\gamma} \gamma \left[F(k\gamma)-F((k-1)\gamma)\right]
=\gamma,\label{14}
\end{align}
since we always have $F^\gamma(x)\ge F(x)$ and $F^\gamma(x)\le F(k\gamma)$ for $x<k\gamma$, $k\in \{1,\dots,n\}$.
Combining inequalities \eqref{15} and \eqref{14}, we conclude that
\begin{equation}
\label{3de}
\bigl|2p\tau-\G_\tau(q^\gamma_\tau)\bigr| =
\bigl|\G_\tau(q_\tau)-\G_\tau(q^\gamma_\tau)\bigr|
\le 3\gamma.
\end{equation}
Consider now two different cases depending on the value of $\tau$:

1. If $\tau>\frac{3}{2} n \gamma$, we have 
\begin{equation}
\G_\tau(q^\gamma_\tau) \ge 2p\tau-3\gamma > 2p\tau-\frac{2\tau}{n} {\ge \G_\tau(q_0-\tau)},
\end{equation}
which means that $q^\gamma_\tau \ge q_0-\tau$. Since we also have $q^\gamma_\tau \le q^\gamma_0+\tau \le q_0+\tau$, it follows  that $q^\gamma_\tau \in (q_0-\tau,q_0+\tau)$.
As in this interval $\G_\tau$ is strictly increasing with a rate of at least $\frac{1}{n}$, the condition \eqref{3de} implies that
\begin{equation}
\bigl|  q^\gamma_\tau - q_\tau\bigr|\le 3n\gamma.
\end{equation}

2. If $\tau \le \frac{3}{2} n \gamma$, we have
\begin{equation}
\bigl|  q^\gamma_\tau - q_\tau\bigr| \le 2\tau+\gamma \le (3n+1)\gamma.
\end{equation}
Overall, we have showed that
\begin{equation}
\bigl|  q^\gamma_\tau - q_\tau\bigr| \le (3n+1)\gamma,
\end{equation}
and to conclude the proof, it is enough to notice that we have $q^\gamma_t = q^\gamma_\tau$ by definition and $\bigl|  q_t - q_\tau\bigr| \le \gamma$ as $t-\tau < \gamma$ and $|\dot q_t|\le 1$.
\end{proof}

\section{Random initial condition}
\label{s5}
Let $\mu$ be a probability measure $\mu$ on $[0,1]$ and $\epsilon=1/n$ for some positive integer $n$. Consider the continuous model of Section \ref{Cont} with initial distribution given by  $\epsilon^{-1}$ independent bids with law $\mu$. We prove a law of large numbers for the max-price and the empirical measure of bids when $\vep$ goes to zero and establish Gaussian fluctuations for the diffusively rescaled max-price $q^\vep_t$, the max-price associated to the empirical distribution $\mu^\vep$, and compute the covariances.
\subsection{Law of large numbers}
Let $p\in[0,1]$ and $\mu (db):= f(b)\, db$ be a probability measure on $[0,1]$ with cumulative distribution function $F:[0,1]\to [0,1]$, $F(x)=\mu([0,x])$.
For $\eps>0$, let $\sB^\eps:= \{b_1,\dots,b_{1/\vep}\}$, where $b_i$ are iid with law $\mu$.
The empirical measure $\mu^\eps$ and cumulative function $F^\vep$ are defined by
\begin{align}
  \label{eq:30}
  \mu^\eps&:= \eps \sum_{b \in \sB^\eps}\, \delta_b,\qquad  F^\eps(x)=\eps \sum_{b\in \sB^\eps} \one\{b\le x\}.
\end{align}
The evolution of $\mu^\vep$ is given by
\begin{align}
  \label{eq:10}
 \mu^\vep\cT_t= (w^\vep_t+\ell^\vep_t) \cdot \delta_{q^\vep_t}+ \vep\sum_{b\in B^\vep} \left(\one\{b+t>q^\vep_t\}\cdot \delta_{b+t}+ \one\{b-t<q^\vep_t\} \cdot \delta_{b-t}\right),
\end{align}
where $q^\vep_t, w^\vep_t,\ell^\vep_t$ satisfy  the system  \eqref{wt}-\eqref{qb0} with initial measure $\mu^\vep$.

The limiting max-price evolution $(q_t)_{t\ge 0}$ is the solution of  the system  \eqref{wt}-\eqref{qb0} for the initial measure $\mu$.  Define
  \begin{align}
    \label{eq:32}
    G^{\eps}(x) = \int_0^x F^\eps(y) dy,\qquad G(x) = \int_0^x F(y) dy.
  \end{align}

  \begin{theorem}[LLN]\label{21}
    Let $(q_t^\eps)_{t\ge 0}$ be the max-price evolution for $\mu^\vep$ and  $(q_t)_{t\ge 0}$ the max-price evolution for $\mu$. Then, for all fixed $t\ge 0$,
  \begin{align}
    \lim_{\eps\to0} q^{\eps}_t &= q_t    \qquad \text{a.s}.\label{qet1}
  \end{align}
\end{theorem}

\begin{proof}
By the law of large numbers, 
 \begin{align}
    F^{\eps}(x) &\toe F(x)  ,\label{fet1}\\
    G^{\eps}(x) &= \eps \sum_{b\in \sB^\eps} (x-b) \one_{b\le x} \toe \int_0^1 (x-z)\, \one_{z\le x}\, f(z) dz =  G(x) \label{eq:33}
   \end{align}
where  both convergences hold a.s. and uniformly in $x\in [0,1]$. 

\noindent
 Let
$ \G_t(x):= G(x+t) - G(x-t) $ and $\G^\eps_t(x):= G^\eps(x+t) - G^\eps(x-t)$, then, as a consequence of \eqref{eq:33}, for all fixed $t \ge 0$,
\begin{equation}
\lim_{\eps\to 0} \sup_{x\in [0,1]}\big| \G^\eps_t(x)-\G_t(x)\big|=0 \quad \text{a.s.}
\end{equation}
Theorem \ref{TeoEx} implies that $(q_t)_{t\ge 0}$ is the unique  solution of the equation $\G_t(q_t)=2pt$, whereas $(q^\eps_t)_{t\ge 0}$ is the unique   solution of the equation $\G^\eps_t(q^\eps_t)=2pt$. Using that both $\G_t$ and $\G^{\eps}_t$ are invertible in $2pt$, we deduce that
\begin{equation}
\lim_{\eps\to0} \big| q^{\eps}_t  -q_t  \big|=\lim_{\eps\to 0}\big| (\G^\eps_t)^{-1}(2pt)-(\G_t)^{-1}(2pt)\big|=0 \quad \text{a.s.}
\end{equation}
from which equation \eqref{qet1} follows.
\end{proof}

\subsection{Max-price Gaussian fluctuations}\label{s52}

We compute the fluctuations around the limiting deterministic trajectory.

\subsubsection{Functional central limit theorems for Poisson process empirical functions}

Define the empirical fluctuations for $F^\vep$ and $G^{\vep}$ by
\[
\zeta^{F^\epsilon}(x) := \epsilon^{-1/2}\big(F^\epsilon(x)-F(x)\big), \qquad
\zeta^{G^\epsilon}(x) := \epsilon^{-1/2}\big(G^\epsilon(x)-G(x)\big).
\]

Let $\xi^F$ be the white noise on $[0,1]$ with control measure $\mu$; i.e., a Gaussian random measure with
\[
\mathbb{E}[\xi^F(I)]=0, \quad \operatorname{Cov}(\xi^F(I),\xi^F(J))=\mu(I\cap J).
\]
where $I,J\subset[0,1]$ are Borel sets. Define
\[
\zeta^F(x) := \xi^F([0,x]), \qquad \zeta^G(x) := \int_0^x \zeta^F(z)\,dz.
\]

\begin{proposition}[Functional CLT for $F^\epsilon$]
\label{prop:F}
Under the above assumptions,
\[
\zeta^{F^\epsilon} \stackrel{\text{law}}{\longrightarrow} \zeta^F \quad \text{in } D[0,1] 
\]
 with respect to the Skorokhod $J_1$ topology. Moreover, $\zeta^F$ is a centered Gaussian process with covariance
\[
\mathbb{E}[\zeta^F(x)\zeta^F(y)] = F(\min(x,y)),
\]
i.e., $\zeta^F(x)=B(F(x))$ for a standard Brownian motion $B$.
\end{proposition}

\begin{proof}
For each $\epsilon>0$, the process $\{\zeta^{F^\epsilon}(x):x\in[0,1]\}$ is a  martingale with quadratic variation
\[
\langle\zeta^{F^\epsilon}\rangle(x)=F(x),
\]
which is deterministic and continuous. The jump sizes are $\epsilon^{1/2}\to0$, so the Lindeberg condition holds. By the Functional Central Limit Theorem for Martingales (see Theorem VIII.3.11 in \cite{JacodShiryaev2003}), $\zeta^{F^\epsilon}$ converges in law in $D[0,1]$ to a continuous Gaussian martingale with quadratic variation $F$. This  identifies the  process  $B(F(\cdot))$. 
\end{proof}

\begin{proposition}[Functional CLT for $G^\epsilon$]
\label{prop:G}
Under the same assumptions,
\[
\zeta^{G^\epsilon} \stackrel{\text{law}}{\longrightarrow} \zeta^G \quad \text{in } C[0,1]
\]
with respect to the supremum norm. Hence, $(\zeta^G(x))_{x\in[0,1]}$ is a centered Gaussian process with covariances
\begin{align}
\operatorname{Cov}(\zeta^G(x),\zeta^G(y))
&= \int_0^1 \psi_x(z)\psi_y(z)\,f(z)\,dz \label{eq:cov1}\\
&= \int_0^1 \mathbf{1}_{\{z\le x\wedge y\}} (x-z)(y-z)\,f(z)\,dz, \label{eq:cov2}
\end{align}
where $\psi_x(z):=(x-z)\mathbf{1}_{\{z\le x\}}$.
\end{proposition}

\begin{proof}
Observe that $\zeta^{G^\epsilon}=\Phi(\zeta^{F^\epsilon})$, where $\Phi:D[0,1]\to C[0,1]$ is the integration map $\Phi(\psi)(x)=\int_0^x\psi(z)\,dz$. The map $\Phi$ is continuous when $D[0,1]$ has the Skorokhod $J_1$ topology and $C[0,1]$ the supremum norm (see Theorem 14.5 in \cite{Billingsley1999}).

From Proposition \ref{prop:F}, $\zeta^{F^\epsilon}\stackrel{\text{law}}{\longrightarrow}\zeta^F$ in $D[0,1]$. By the Continuous Mapping Theorem (Theorem 2.7 in \cite{Billingsley1999}),
\[
\zeta^{G^\epsilon}=\Phi(\zeta^{F^\epsilon})\stackrel{\text{law}}{\longrightarrow}\Phi(\zeta^F)=\zeta^G\quad\text{in }C[0,1].
\]
The covariance formula follows by direct computation:
\[
\operatorname{Cov}(\zeta^G(x),\zeta^G(y))=\int_0^x\int_0^y F(\min(z,w))\,dz\,dw
=\int_0^{x\wedge y}(x-z)(y-z)f(z)\,dz.
\]
\end{proof}

\begin{remark}
Proposition \ref{prop:F} gives the best possible convergence for $\zeta^{F^\epsilon}$ because it has jumps and cannot converge in $C[0,1]$. Proposition \ref{prop:G} strengthens to $C[0,1]$ because $\zeta^{G^\epsilon}$ is continuous (in fact $C^1$) and the integration map regularizes the process.
\end{remark}


\begin{corollary}[Fluctuation field for $\mathcal{G}_t$]
\label{cor:mathcalG}
Define for $x\in[0,1]$ and $t\ge 0$:
\[
\mathcal{G}_t(x) := G(x+t)-G(x-t),\qquad 
\mathcal{G}^\epsilon_t(x) := G^\epsilon(x+t)-G^\epsilon(x-t),
\]
with $G$ and $G^\epsilon$ extended constantly outside $[0,1]$. Let
\[
\zeta^{\mathcal{G}^\epsilon}(x,t) := \epsilon^{-1/2}\big(\mathcal{G}^\epsilon_t(x)-\mathcal{G}_t(x)\big).
\]
Then, as $\epsilon\to0$,
\[
\zeta^{\mathcal{G}^\epsilon}(x,t) \stackrel{\text{law}}{\longrightarrow} \zeta^{\mathcal{G}}(x,t) := \zeta^G(x+t) - \zeta^G(x-t)
\]
in the sense of finite-dimensional convergence (or in $C([0,1]\times[0,T])$ for any fixed $T<\infty$). The limit is a centered Gaussian process with covariance
\[
\operatorname{Cov}\big(\zeta^{\mathcal{G}}(x,s),\zeta^{\mathcal{G}}(y,t)\big) = \int_0^1 \Psi_{x,s}(z)\Psi_{y,t}(z)\,\mu(dz),
\]
where $\Psi_{x,s}(z) := (x+s-z)_+ - (x-s-z)_+$, with $(\cdot)_+=\max(\cdot,0)$.
\end{corollary}

\begin{proof}
By definition we have
\[
\zeta^{\mathcal{G}^\epsilon}(x,t) = \epsilon^{-1/2}\Big(\big[G^\epsilon(x+t)-G(x+t)\big] - \big[G^\epsilon(x-t)-G(x-t)\big]\Big)
= \zeta^{G^\epsilon}(x+t) - \zeta^{G^\epsilon}(x-t).
\]

\noindent
For fixed $t\ge0$, define $\Phi_t : C[0,1] \to C[0,1]$ by
\[
\Phi_t(\psi)(x) := \psi(x+t) - \psi(x-t),
\]
where $\psi$ is extended constantly outside $[0,1]$. The map $\Phi_t$ is continuous with respect to the supremum norm because, for any $\psi_1,\psi_2\in C[0,1]$,
\[
\|\Phi_t(\psi_1)-\Phi_t(\psi_2)\|_\infty \le 2\|\psi_1-\psi_2\|_\infty.
\]

\noindent
From Proposition \ref{prop:G}, $\zeta^{G^\epsilon}\stackrel{\text{law}}{\longrightarrow}\zeta^G$ in $C[0,1]$. By the Continuous Mapping Theorem, for each fixed $t$,
\[
\zeta^{\mathcal{G}^\epsilon}(\cdot,t) = \Phi_t(\zeta^{G^\epsilon}) \stackrel{\text{law}}{\longrightarrow} \Phi_t(\zeta^G) = \zeta^{\mathcal{G}}(\cdot,t) \quad\text{in }C[0,1].
\]
For joint convergence in $(x,t)$, consider any finite collection $(x_i,t_i)$. The map
\[
\psi \mapsto \big( \psi(x_i+t_i)-\psi(x_i-t_i) \big)_{i=1}^k
\]
is continuous from $C[0,1]$ to $\mathbb{R}^k$, so finite-dimensional convergence follows.

\noindent
Since $\zeta^G$ is Gaussian, $\zeta^{\mathcal{G}}$ is Gaussian with mean zero. Its covariance is
\begin{align*}
&\operatorname{Cov}\big(\zeta^{\mathcal{G}}(x,s),\zeta^{\mathcal{G}}(y,t)\big) \\
&= \mathbb{E}\big[ (\zeta^G(x+s)-\zeta^G(x-s)) (\zeta^G(y+t)-\zeta^G(y-t)) \big] \\
&= \mathbb{E}[\zeta^G(x+s)\zeta^G(y+t)] - \mathbb{E}[\zeta^G(x+s)\zeta^G(y-t)] \\
&\quad - \mathbb{E}[\zeta^G(x-s)\zeta^G(y+t)] + \mathbb{E}[\zeta^G(x-s)\zeta^G(y-t)].
\end{align*}
From Proposition \ref{prop:G}, $\mathbb{E}[\zeta^G(u)\zeta^G(v)] = \int_0^1 (u-z)_+ (v-z)_+ f(z)\,dz$. Therefore
\begin{align*}
\operatorname{Cov}\big(\zeta^{\mathcal{G}}(x,s),\zeta^{\mathcal{G}}(y,t)\big)
&= \int_0^1 \Big[ (x+s-z)_+ - (x-s-z)_+ \Big] \Big[ (y+t-z)_+ - (y-t-z)_+ \Big] f(z)\,dz \\
&= \int_0^1 \Psi_{x,s}(z) \Psi_{y,t}(z) \,\mu(dz),
\end{align*}
where $\mu(dz)=f(z)dz$. 
\end{proof}

%

\subsubsection{Central limit theorem  for the max-price}

We now apply the functional Delta Method to obtain the limit distribution of the inverse process.

\begin{proposition}[Delta Method for Quantiles]
\label{thm:quantile}
Assume that for each fixed $t\ge0$, the map $x\mapsto\mathcal{G}_t(x)$ is continuously differentiable and strictly increasing on an interval containing $[0,1]$. Then, for any $y$ in the interior of the range of $\mathcal{G}_t$,
\[
\epsilon^{-1/2}\Bigl( (\mathcal{G}^\epsilon_t)^{-1}(y) - \mathcal{G}_t^{-1}(y) \Bigr) \stackrel{\text{law}}{\longrightarrow} -\,\frac{\zeta^{\mathcal{G}}\bigl(\mathcal{G}_t^{-1}(y), t\bigr)}{\frac{d}{dx}\mathcal{G}_t\bigl(\mathcal{G}_t^{-1}(y)\bigr)}.
\]
\end{proposition}

\begin{proof}
This is a direct application of Theorem 3.9.4 in \cite{vdVaartWellner1996}, which states that if $\mathbb{G}_n \xrightarrow{\text{law}} \mathbb{G}$ in a normed space and $\phi$ is Hadamard differentiable at the limit point, then $\phi(\mathbb{G}_n) \xrightarrow{\text{law}} \phi'_{\mathbb{G}}(\mathbb{G})$.

Here we take $\mathbb{G}_n = \zeta^{\mathcal{G}^\epsilon}(\cdot, t)$ (as a process in $x$) and $\phi$ the inverse map $\phi(f) = f^{-1}$. For a fixed $t$, the map $f \mapsto f^{-1}$ is Hadamard differentiable at $\mathcal{G}_t$ with derivative
\[
\phi'_{\mathcal{G}_t}(h)(y) = -\,\frac{h(\mathcal{G}_t^{-1}(y))}{\frac{d}{dx}\mathcal{G}_t(\mathcal{G}_t^{-1}(y))},
\]
provided $\frac{d}{dx}\mathcal{G}_t > 0$ (see Lemma 3.9.23 in \cite{vdVaartWellner1996}). The result follows by applying the functional Delta Method to $\zeta^{\mathcal{G}^\epsilon}(\cdot, t) \xrightarrow{\text{law}} \zeta^{\mathcal{G}}(\cdot, t)$. 
\end{proof}

\begin{theorem}
\label{cor:qprocess}
Under the assumptions of Theorem~\ref{thm:quantile},
\[
\zeta^{q,\epsilon}_t := \epsilon^{-1/2}\bigl(q_t^\epsilon - q_t\bigr) \stackrel{\text{law}}{\longrightarrow} \zeta^q_t := -\,\frac{\zeta^{\mathcal{G}}\bigl(q_t, t\bigr)}{\frac{d}{dx}\mathcal{G}_t(q_t)}.
\]
Since $\frac{d}{dx}\mathcal{G}_t(x) = G'(x+t) - G'(x-t) = F(x+t) - F(x-t)$, we have
\[
\frac{d}{dx}\mathcal{G}_t(q_t) = F(q_t+t) - F(q_t-t) =: c_t.
\]
Thus
\[
\zeta^q_t = -\,\frac{\zeta^{\mathcal{G}}(q_t, t)}{c_t}.
\]
The process $\{\zeta^q_t : t\ge0\}$ is centered Gaussian with covariance
\[
\operatorname{Cov}(\zeta^q_t, \zeta^q_s) = \frac{1}{c_t c_s} \int_0^1 \Psi_{q_t,t}(z) \Psi_{q_s,s}(z) \,\mu(dz),
\]
where $\Psi_{x,t}(z) = (x+t-z)_+ - (x-t-z)_+$ and $\mu(dz)=f(z)dz$.
\end{theorem}

\begin{proof}
We now  specialize Theorem~\ref{thm:quantile}  to $y = 2pt$, where $p\in[0,1]$ is fixed. Define
\[
q_t := \mathcal{G}_t^{-1}(2pt), \qquad q_t^\epsilon := (\mathcal{G}^\epsilon_t)^{-1}(2pt).
\]

\[
\epsilon^{-1/2}(q_t^\epsilon - q_t) \stackrel{\text{law}}{\longrightarrow} -\,\frac{\zeta^{\mathcal{G}}(q_t, t)}{\frac{d}{dx}\mathcal{G}_t(q_t)}.
\]
The expression for the derivative follows from $\mathcal{G}_t(x)=G(x+t)-G(x-t)$ and $G'=F$. The covariance follows from Corollary~\ref{cor:mathcalG}:
\[
\operatorname{Cov}(\zeta^{\mathcal{G}}(q_t,t), \zeta^{\mathcal{G}}(q_s,s)) = \int_0^1 \Psi_{q_t,t}(z) \Psi_{q_s,s}(z) \,\mu(dz).
\]
Scaling by $1/(c_t c_s)$ gives the covariance of $\zeta^q_t$. 
\end{proof}

%

\section{Multi-velocity model}
\label{s6}
We generalize the discrete model by accepting different individual updating policies. Each agent $i$ is associated to the usual initial bid $b_i$ and  a ``velocity'' $v_i\in\mathbb{N}$. In the discrete case, after each round, agent $i$ adds/subtracts $\gamma v_i$ to its previous bid, if it was winner/loser in the previous round. Our main result here is that the max-price velocity is proportional to the harmonic mean of the velocities associated to the bids at the max-price.

We define  a discrete version of the model and prove the result when there are only two possible velocities,  then define a continuous version, and prove the convergence of the discrete model to the continuous one. 

\subsection{Discrete-time multi-velocity model}
Let $n$ be the number of bids, $\gamma$ be the learning parameter and $p\in [0,1]$ such that $np\in \NN$, the winner fraction. If the bid $(b,v)$ wins, then it is updated to $(b+\gamma v,v)$; if the bidder loses, it is updated to $(b-\gamma v,v)$. If  $v$ equals $1$ for all $b$, we get the previous model.

Let $B=[(b_1,v_1), \ldots, (b_n,v_n)]$ be the initial bid-velocity multi-set, assuming ordered $b_i\in I^\gamma$ and $v_i>0$. 
The single time-step evolution  ${T}^{\gamma} B$ of $B$  is defined as follows. Let
\begin{align}
  \label{tb1}
  b'_i&:= b_i+\gamma v_i \bigl(\one\{0<i\le np\} - \one\{np<i\le n\}\bigr) , \qquad i=1, \ldots, n
\end{align}
Let $\sigma=(\sigma_1, \ldots, \sigma_n)$ be a permutation of $\{1, \ldots, n\}$  that satisfies $b'_{\sigma_1}\le\dots\le  b'_{\sigma_n}$. Define 
\begin{align}
  \label{6}
  {{T}}^{\gamma} B := [(b'_{\sigma_1}, v_{\sigma_1})\ldots, (b'_{\sigma_n}, v_{\sigma_n})].
\end{align}
Each bidder keeps its original velocity parameter after the reordering.
The max-price associated to $B$ is  $q(B)= b_{np}$.
The evolution at the $\kappa$-th time-step is given by $(\rmT^{\gamma})^\kappa B$,  and  the max-price evolution is given by $q\bigl((\rmT^{\gamma})^\kappa B)$.

\paragraph{Velocity of the max-price assuming periodicity} We will assume that there is an initial multi-set that is translated after a finite number of steps, meaning the motion is periodic. 
The next result says that in this case the displacement of the max-price at the end of the period is proportional to the harmonic mean of the velocities.

\begin{proposition}[Max-price velocity is a harmonic mean of velocities]\label{i_vel}
Consider the evolution \eqref{6} and assume it is {\em periodic}, in the sense  that there exist $N\in\mathbb{Z},\kappa\in\mathbb{N}$ such that, for any initial configuration $B$,
\begin{align}
(\rmT^{\gamma})^\kappa(B) &=B+N\gamma:= [(b_1+N\gamma, v_1),\dots, (b_n+N\gamma, v_n)].
   \end{align}
Then we have
\begin{equation} \label{vel}
\frac{N}{\kappa}=\frac{(w-\ell)}{\sum_{i=1}^n \frac{1}{v_i}}.
\end{equation}
\end{proposition}

\begin{proof}
Denote by $w_i(\kappa)$ and $\ell_i(\kappa)$ the number of times the $i$-th bidder respectively wins and loses the auction in the first $\kappa$ steps.
In order for the hypothesis to be satisfied, the following identities must hold for the $i-th$ bidder:
\begin{equation}
\begin{cases}
(w_i(\kappa)-\ell_i(\kappa))v_i=N \\
w_i(\kappa)+\ell_i(\kappa)=\kappa.
\end{cases}
\end{equation}
Solving the system for $w_i$ and $\ell_i$ we get
\begin{equation}
w_i(\kappa)=\frac{1}{2}\left(\kappa+\frac{N}{v_i}\right), \qquad
\ell_i(\kappa)=\frac{1}{2}\left(\kappa-\frac{N}{v_i}\right).
\end{equation}
As in each of the $\kappa$ steps $w$ bidders win and $\ell$ lose, we have that overall it must hold
\begin{align}
\kappa w=\sum_{i=1}^n w_i(\kappa)=\sum_{i=1}^n \frac{1}{2}\left(\kappa+\frac{N}{v_i}\right) = \frac{n\kappa}{2}+\frac{N}{2}\sum_{i=1}^n \frac{1}{v_i}
\\
\kappa\ell=\sum_{i=1}^n \ell_i(\kappa)=\sum_{i=1}^n \frac{1}{2}\left(\kappa-\frac{N}{v_i}\right) = \frac{n\kappa}{2}-\frac{N}{2}\sum_{i=1}^n \frac{1}{v_i}.
\end{align}
Subtracting the previous two equations we finally obtain
\[
(w-\ell)\kappa=N\sum_{i=1}^n \frac{1}{v_i} \implies
\frac{N}{\kappa}=\frac{(w-\ell)}{\sum_{i=1}^n \frac{1}{v_i
}}.\qedhere
\]
\end{proof}

\begin{remark}\rm
Defining the harmonic mean  of the actual velocities $v_i$ by
\begin{align}
  h:= \frac{n}{\sum_{i=1}^n \frac{1}{v_i}}.
\end{align}
the equation \eqref{vel} can be written as
\begin{equation}
\frac{N}{\kappa}
= \frac{w-\ell}{w+\ell} \cdot h;
\end{equation}
we can see that the velocity at which the bidders are travelling is given by the velocity that they would have if   all the velocity parameters were set to $1$ multiplied by the harmonic velocity $h$. 
\end{remark}

\vskip.3cm
\begin{lemma} \label{group}
\noindent
\begin{itemize}
\item[i)] Two bidders $b$ and $b'$ with the same velocity parameter $v$  and initial distance $c > 0$ will always stay at a distance of at most $ \max \left\lbrace c, 2\gamma v \right\rbrace $.
\item[ii)] Let $v_{\max}:=\max\{v_1, \ldots, v_n\}$  and suppose that there is a  bidder that has velocity parameter $v_{\max}$ that is distant at most $2\gamma v_{\max}$ from the max-price. Then this bidder  will always stay at a distance of at most $2\gamma v_{\max}$ from the max-price.
\end{itemize}
\end{lemma}

\begin{proof}
We will show that, in both cases, the bound on the distance stays true after one step. Then, by induction it follows that  the bound remains true at all times. 
\begin{itemize}
\item[i)] If the two bidders are both above or below the max-price then at the following step they will obviously keep the same distance.  If one of the two bidders is above and the other is below the max-price, then  the one below moves up by $\gamma v$ and the one above moves down by $\gamma v$. If $c\ge 2\gamma v$ the distance then decreases by $\gamma v$, otherwise at the following step they will have exchanged their relative order, but, since each of them moves by $v\gamma$, their distance will still be less than $2v\gamma$.
\item[ii)] First of all we notice that at each step the max-price cannot change by more than $\gamma v_\text{max}$. Suppose without loss of generality that the bidder is below the max-price, then it is a winner and then it  will move up by $\gamma v_\text{max}$.  As a consequence, the difference between  the max-price and the bidder position (which is positive) cannot increase, and, given the bound on the max-price velocity, also its absolute value will be less than $2\gamma v_\text{max}$.
\end{itemize}
\end{proof}

\subsection{Two velocities}
\label{s62}

Let $B=[(b_1,v_1),\dots,(b_n,v_n)]$ be a multi-set with $b_i \in I^\gamma$, and assume that there are only two velocities, $v_i\in\{\nua,\nub\}$, with $\nub> \nua>0$. Denote $B_\pm=[(b_i,v_i): v_i =\nu_{\pm}]$, so that $B=B_+\cup B_-$. Let $n_\pm:=|B_\pm|$, the cardinality of the multi-sets. 
Denote by $n= \na+\nb$ the total number of bidders, and by $w$ and $\ell$  the total number of winners and losers, respectively. 

We show  now that if  the bidders  have periodic increments,  then they travel with speed
\begin{equation}\label{harmo}
\frac{(w-\ell)}{n}\cdot h, \quad h:= \frac{n}{ \frac{\na}{\nua}+\frac{\nb}{\nub} }.
\end{equation}
Notice that $h$ is the harmonic mean of the bid-velocities.

In the following proposition we show that the system can attain periodic increments if and only if the absolute value of the global velocity \eqref{harmo} is less or equal than $\nua$.

\begin{proposition}\label{Sep} Let $\nua<\nub$ and consider $B=B_+\cup B_-$. 
 Then  the   two statements below are equivalent:
\begin{itemize}
\item  the motion attains a periodic increment state, in the sense that  there exist $k_0, K, N \in \mathbb N$ such that
\begin{equation}
  (\rmT^{\gamma})^{k_0+K} B= (\rmT^{\gamma})^{k_0} B+\gamma N,
\end{equation}
\item moreover
\begin{equation} \label{2per_con}
  \frac{|w-\ell|}{n}\cdot h
 \le \nua.
\end{equation}
\end{itemize}
\end{proposition}
\begin{proof}
{\bf Step I.} We first show that condition \eqref{2per_con} implies periodicity. To this end we prove that, if \eqref{2per_con} holds true, then  at all times,  the maximum distance between all the bids remains bounded by $c\gamma$ for some constant $c>0$, namely
\begin{equation} \label{intermezzo}
\sup_{k \in \mathbb N}\;\sup_{b,b'\in (T^{\gamma})^{ k}B} \frac 1 \gamma |b-b'| < c \;.
\end{equation}
As a consequence,  their relative configuration must eventually repeat within a finite time, and therefore there is a time-step such that the configuration starting from it has a motion which is necessarily periodic.

Abusing notation, define:
\begin{equation*}
C:=\max_{b,b'\in  B} \frac 1 \gamma |b-b'| 
\qquad
C_+:=\max_{b,b'\in B_+} \frac 1 \gamma |b-b'| 
\qquad
C_-:=\max_{b,b'\in B_-} \frac 1 \gamma |b-b'| 
\end{equation*}

To prove the bound on the distance between the bids in equation \eqref{intermezzo},
 it is sufficient to prove that each bid remains within a distance of order $\gamma$ from the max-price. For what concerns the fast bids with velocity $\nub$, we know from item ii) of Lemma \ref{group} with $v_{\max} = \nub$ that they remain at a distance of at most $\max \{2\nub \gamma, C\gamma\}$ from the max-price.  We are left to prove the claim for the bidders with velocity equal to $\nua$.

 If at a given time-step some of these slow bidders are winning and some others are losing, the max-price must necessarily be between  the lowest $b_l$ and  the highest $b_h$ of them. Therefore,  thanks to item i) of Lemma \ref{group},  the distance between these two bidders is bounded by $ \max \left\lbrace C_-\gamma ,2\nua\gamma \right\rbrace $ at all times.  As a consequence,   all the other slow bidders remain within the same distance from the max-price.
At all these timesteps the maximum distance between the bidders is then bounded by $c_1=\max \{2\nub \gamma, C\gamma\}+\max \left\lbrace C_-\gamma ,2\nua\gamma \right\rbrace$.

We now define an excursion as a collection of consecutive timesteps where all the slow bidders are winners (or losers).
At this point it is sufficient to control how far the bidders can go from the max-price during each excursion.
Consider now the first timestep of an excursion where all the slow bidders are winners and call $x$ the price bid by the lowest of the slow bidders and $y$ the price bid by the highest fast bidders. We then have that after $i$ timesteps of the excursion the value $x+i\nua\gamma$ is a lower bound for the price bid by the slow bidders. Considering now the fast bidders we know that for all the duration of the excursion exactly $w-n_a$ of them will be winners and $\ell$ losers at each timestep, and therefore the position of their barycenter will increase of $\frac{w-n_a-\ell}{n_b}\nub\gamma$. Observing that at the initial step of the excursion the position of the barycenter can be at most $y$ and that at each step each fast bidder differs by at most $\max \{2\nub \gamma, C_+\gamma\}$ from the barycenter we obtain that after $i$ timesteps of the excursion the value $y+\max \{2\nub \gamma, C_+\gamma\}+i\frac{w-n_a-\ell}{n_b}\nub\gamma$ is an upper bound for the price bid by the fast bidders. Condition \eqref{2per_con} implies that $\frac{w-n_a-\ell}{n_b}\nub\gamma\le\nua$ and as a consequence during the excursion the maximum distance between the bidders is bounded by $y-x+\max \{2\nub \gamma, C_+\gamma\}$. To conclude we now have to show that we have a bound on $y-x$, and that is true because either the first step of the excursion is the starting configuration, and thus the distance is bounded by $C\gamma$ or at the previous step at least one of the slow bidders was a loser, which implies that at the first step of the excursion we have $y-x\le \max \left\lbrace C_-\gamma-\nua\gamma, \nua\gamma \right\rbrace+ \max \left\lbrace C\gamma-\nub\gamma, \nub\gamma \right\rbrace$

The case in which all the slow bidders are losers can be analysed  in a similar way. This concludes the proof of {\em Step I}.

{\bf Step II.} Now we prove that if \eqref{2per_con} is not satisfied then  the two groups of bidders will separate (see Figure \ref{figura7}) and  the motion is not periodic. Let us suppose that $w>l$, as the proof for the other case is analogous. Consider   the barycenter of fast  bidders,  since at least $w-n_a$ of them are winning, it follows that their barycenter increment at each time-steps is at least $\frac{w-n_a-\ell}{n_b}\nub\gamma$ which by hypothesis is greater than $\nua\gamma$. As a consequence after a finite number of steps all the slow bidders will be below the max-price and move with velocity $\nua$ increasing their distance from the fast bidders and preventing any periodicity.
\end{proof}

\begin{figure}[th]
  \centering
  \includegraphics[width= .7\textwidth]{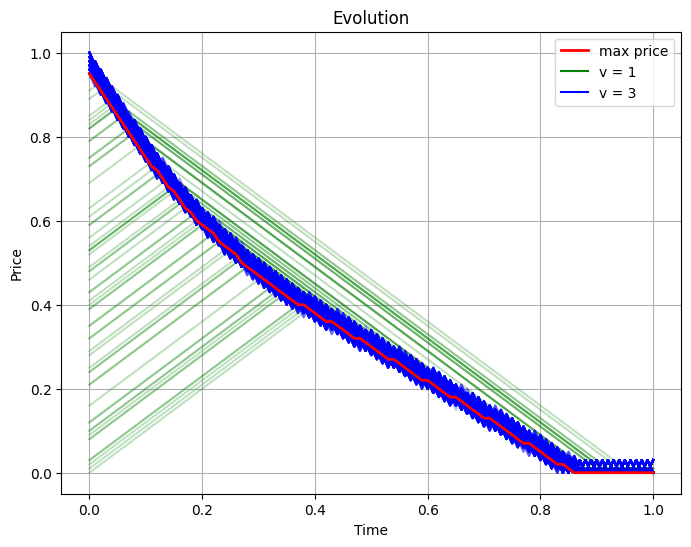}
  \caption{Two different velocities: $\nu_{-}=1$ for $n_- = 60$ bids initially distributed uniformly at random in $[0, 0.95]$ and $\nu_{+}=3$ for $n_+ = 240$ bids initially distributed uniformly at random in $[0.95, 1]$ with $p=0.3$ and $\gamma = 0.01$.  }\label{figura7}
  \label{2velpas}
\end{figure}

\begin{remark} Notice that, if  condition  \eqref{2per_con} is not satisfied, the evolution is not globally periodic. Nevertheless, in any case, the bids in the sets $B_-$ and  $B_+$, taken separately,  perform a periodic motion in the sense of  Lemma \ref{pe4}.
\end{remark}

\begin{theorem}[Ballistic price evolution] \label{ball1}
Let $p\in(0,1)$ and consider a collection of bid-velocity multi-sets  $\{B^{\gamma}: \gamma \in (0,1) \}$, each with only two velocities $\nu_-,\nu_+$, satisfying the following condition:
\begin{equation} \label{ottantuno}
\sup_{\gamma \in (0,1)}\sup_{b,b'\in  B^{\gamma}} \frac 1 \gamma |b-b'| <\infty.
\end{equation}
Moreover suppose that the initial discrete max-price $q_0^\gamma:= q(B^{\gamma})$ is such that
\begin{equation}\label{q0}
 \sup_{\gamma \in (0,1)}  \,  \frac 1 \gamma|  q_0^\gamma - q_0|<\infty
\end{equation}
for some  $q_0\in[0,1]$. Recall the harmonic velocity 
\begin{align}
  h=h(\nu_-,\nu_+,n_-,n_+) := \displaystyle{\frac{n}{\frac{\na}{\nua}+\frac{\nb}{\nub}}}.
\end{align}
Let $\rmT^{\gamma}_tB^\gamma:= (\rmT^\gamma)^{\lfloor t/\gamma\rfloor}B^\gamma$, see \eqref{6}, and $q^\gamma_t:= q\bigl(\rmT^{\gamma}_tB^\gamma)$. 
 Then,
   \begin{equation}\label{qt}
     \sup_{\gamma \in (0,1)}  \sup_{t\ge 0} \,  \frac 1 \gamma|  q_{t}^\gamma - q_t |<\infty,
  \end{equation}
  where
  \begin{align}
         q_t :=
        \begin{cases} 
          q_0 + \displaystyle{\frac{(w-\ell)}{n}\,h\,t} \quad
          & \text{if} \quad \displaystyle{\frac{|w-\ell|}{n}\,h\le \nua}
        \\[3mm]
          q_0 + \displaystyle{\frac{w-\ell-\na\sign(w-\ell)}{\nb}\,\nub t} \quad
          & \text{if} \quad \displaystyle{\frac{|w-\ell|}{n}\,h> \nua}
        \end{cases} \;.
  \end{align}
Moreover  denote by  $b_i^{\gamma}(t)$,  the $i$-th bid-component of the ordered multi-set $T^{\gamma}_tB^\gamma$.
Then
\begin{equation}\label{ping}
\sup_{\gamma\in (0,1)} \sup_{t\ge 0} \frac 1 \gamma |b_i^{\gamma}(t)-b_i(t)|<\infty \qquad \text{for all } i=1, \ldots, n
\end{equation}
with
 \begin{align}\label{pong}
     b_i(t) :=
        \begin{cases} 
          q_t \quad
          & \text{if} \quad \displaystyle{\frac{|w-\ell|}{n}\,h\le \nua} \: \text{ or } \: v_i =\nub
        \\[3mm]
          q^0 +   \sign(w-\ell) \nua t \quad
          & \text{if} \quad \displaystyle{\frac{|w-\ell|}{n}\,h> \nua} \: \text{ and } \: v_i =\nua
        \end{cases} \;.
  \end{align}
\end{theorem}
\begin{proof}
{\bf Case 1)} First consider the case in which the condition \eqref{2per_con} is satisfied.  Using Proposition \ref{Sep} we have that for each $\gamma$ exist $k_0^{\gamma}, K^{\gamma}, N^{\gamma} \in \mathbb N$ such that
\begin{equation}
(T^\gamma)^{k_0^{\gamma}+K^{\gamma}} B^{\gamma}= (T^\gamma)^{k_0^{\gamma}} B^{\gamma}+\gamma N^{\gamma},
\end{equation}
as a consequence  we have that 
\begin{equation}
\sup_{t\ge 0} \,  \frac 1 \gamma|  q_t^\gamma - q_t |
=\sup_{0\le t\le k_0^{\gamma}+K^{\gamma}} \,  \frac 1 \gamma|  q_t^\gamma - q_t |.
\end{equation}

\noindent
From \eqref{ottantuno}
 we know that there exists $c>0$ such that
  \begin{equation}\label{d}
\sup_{\gamma \in (0,1)}\sup_{b,b'\in B^{\gamma}} \frac 1 \gamma |b-b'| \le c.
  \end{equation}
We have that  there exists a finite number of configurations satisfying this condition.

 For any fixed configuration  $B \in \{B^\gamma: \, \gamma \in (0,1)\}$, we define the related {\em centred configuration}, i.e. the collection of  bid distances from the max-price divided by $\gamma$ together with their velocities: 
\begin{equation}
\Pi B :=\left( \left( \tfrac 1 {\gamma} \left(b_1-q_0^{\gamma}\right), v_1\right), \ldots, \left( \tfrac 1 {\gamma} \left(b_n-q_0^{\gamma}\right), v_n\right) \right) 
\end{equation}
then we have that the values  $k_0^{\gamma}$ and $k^{\gamma}$ are uniquely  determined the centred initial configuration. This means that for each couple of configurations $B^{\gamma_1}$ and $B^{\gamma_2}$ such that $\Pi B^{\gamma_1}=\Pi B^{\gamma_2}$  one has that $k_0^{\gamma_1}= k_0^{\gamma_2}$ and $k^{\gamma_1}=k^{\gamma_2}$.

\noindent
Since  there exists a finite number of possible  configurations satisfying  \eqref{d}, we have that the projection $\Pi$ maps  the set  $\{B^{\gamma} :\gamma \in (0,1) \}$  to  a finite set. 
As a consequence we have that
     $\sup_{\gamma \in (0,1)} k_0^{\gamma}< \infty$ and $\sup_{\gamma \in (0,1)} k^{\gamma}< \infty$.

From the fact that  both $q_t^\gamma$ and  $q_t$ travel at  velocity $\le \nub$ in absolute value,
we have that
\begin{eqnarray}
  \sup_{0\le t\le (k_0^{\gamma}+K^{\gamma})} \,  \frac 1 \gamma|  q_t^\gamma - q_t | & \le &  \frac 1 \gamma|  q_{0}^\gamma - q_0 |+ 2 \nub (k_0^{\gamma}+K^{\gamma})
\end{eqnarray}
and then, from the hypothesis \eqref{q0} we can conclude \eqref{qt}. 
For what concerns the bids we have that
\begin{equation}\label{Qquo}
\sup_{\gamma\in (0,1)} \sup_{ t \ge 0 } \frac 1 \gamma |b_i^{\gamma}(t)-b_i(t)| \le
\sup_{\gamma\in (0,1)} \sup_{t \ge 0 } \frac 1 \gamma \left( |b_i^{\gamma}(t)-q_t^\gamma| + | q_t^\gamma-q_t|\right).
\end{equation}
The r.h.s. of the bound above is bounded as in  \eqref{qt}. Whereas
\begin{equation}\label{Qqui}
\sup_{\gamma\in (0,1)} \sup_{t \ge 0} \frac 1 \gamma  |b_i^{\gamma}(t)-q_t^\gamma| <\infty
\end{equation}
follows from the considerations done in the proof of {\em Step I} of Proposition \ref{Sep}.

\noindent
{\bf  Case 2)} In case condition \eqref{2per_con} is not satisfied, there exists a finite number of time-steps, say  $\bar k^{\gamma}$, such that  slow particles are strictly below or strictly above the max-price 
for all times $t \ge \gamma \bar k^{\gamma}$. For the same kind of reasoning used in the proof of Case 1) we can argue that $\sup_{\gamma \in (0,1)} \bar k^{\gamma}<\infty$. 
For $t \ge \gamma \bar k^{\gamma}$  slow particles do not interact anymore with the evolution of the max-price. 
And then, after a certain number of steps $k_0^{\gamma}>\bar k^{\gamma}$, fast particles organize themselves in a periodic  motion of period $k^{\gamma}$, and determine the evolution of the max-price.  Then for fast particles  we can repeat the same argument used for the proof of Case 1). This concludes the proof of \eqref{qt}.

\noindent
For what concerns the bids, we have that, if $b_i$ is a fast bid, i.e. if $v_i =\nub$,  using \eqref{Qquo} we have that, once proven \eqref{qt} it remain to show \eqref{Qqui}. From item ii) of Lemma \ref{group} we know that, even in this case, fast bids  remain close to the interface throughout the dynamics and then \eqref{Qqui} holds true.

\noindent 
For  what concerns slow bids, we know that, for all  $t \ge \gamma \bar k^{\gamma}$, they are strictly below or strictly above $q_t$ and travel at velocity $\sign(w-\ell) \nua$.
Since from the hypothesis at time zero we know that
\begin{equation}
\sup_{\gamma\in (0,1)}\frac 1 \gamma  |b_i^{\gamma}-q_0| <\infty
\end{equation}
and, since 
\begin{equation}
\sup_{\gamma\in (0,1)}\sup_{t \in [0, \gamma \bar k^{\gamma}]}\frac 1 \gamma  |b_i^{\gamma}(t)-q_0| <\infty
\end{equation}
and, as a consequence, also
\begin{equation}
\sup_{\gamma\in (0,1)}\sup_{t \ge \gamma \bar k^{\gamma}}\frac 1 \gamma  |b_i^{\gamma}(t)-q_0- \sign(w-\ell)\nua t| <\infty
\end{equation}
then, since  $\sup_{\gamma \in (0,1)} \bar k^{\gamma}<\infty$, it follows that
\begin{equation}
\sup_{\gamma\in (0,1)}\sup_{t \ge 0}\frac 1 \gamma  |b_i^{\gamma}(t)-q_0- \sign(w-\ell)\nua t| <\infty.
\end{equation}
This concludes the proof.
\end{proof}

\subsection{Space-time continuous multi-velocity model}
Here we discuss a continuous version of the multi-velocity model with initial sum of delta measures.

Let $B=[(b_1,v_1),\dots,(b_n,v_n)]$ be a multi-set with ordered elements $b_i\in[0,1]$ and $v_i \in \mathbb R^+$.
Denote by $\mu= \sum_{i=1}^n \frac1n\delta_{(b_i,v_i)}$ the measure that assigns a mass $1/n$ to each bid. The cumulative mass function $F$ associated to $B$ is given by
\begin{align}
  F(x)&=\frac1n\sum_{i=1}^n \one\{b_i \le x\}
\end{align}
For simplicity we consider only two velocities, $v_i\in\{\nua,\nub\}$ with $\nua<\nub$, and we denote by $B^\pm:= [(b,v)\in B: v= \nu^\pm] $; $B=B^+\cup B^-$. 
As usual we denote by $p\in(0,1)$ the parameter determining the fraction of mass that is sold and 
recall $q_0=F^{-1}(p)$, the max-price at time zero. 

We will define a system of equations for the max-price $q_t$ and bid dynamics 
\[
B_t:=[(b_1(t),v_1), \ldots, (b_n(t), v_n)]
,\]
 where $b_i(t)$ is the reordered position at time $t$ of the $i$-th bid. Each bid conserves its original velocity parameter throughout the dynamics. 
The cumulative mass, the max-price and the winning and losing masses at the max-price at time $t$ are denoted by
\begin{gather}
F_t(x):= \frac1n\sum_{i=1}^n  \one \{b_i(t)\le x\}\\
 q_t :=F^{-1}_t(p),\qquad w_t=p-F_t(q_t^-),\qquad 
\ell_t=F_t(q_t)-p
\end{gather}
Consider the system of equations \eqref{eq:h_t}-\eqref{eq:init_cond} for $(q_t,B_t)_{t\ge 0}$: 
\begin{align}\label{eq:h_t}
  h_t&= \frac{w_t+\ell_t}{\frac1n\sum_{i=1}^n \frac{1}{v_i}\one\{b_i(t)=q_t\}}\\[2mm]
  u_t &=
       \begin{cases}
         \displaystyle{\frac{w_t-\ell_t}{w_t+\ell_t} \cdot h_t},
         &\displaystyle{\Big|\frac{w_t-\ell_t}{w_t+\ell_t} \cdot h_t\Big|\le \nua} \\
         &\vspace{-1mm}\\
         \displaystyle{\frac{w_t-\ell_t-\sign(w_t-\ell_t)\, m^-_t}
         {m^+_t}\cdot \nub}
         &\displaystyle{\Big|\frac{w_t-\ell_t}{w_t+\ell_t} \cdot h_t \Big| > \nua}
       \end{cases} \label{eq:u_t_def}\\[2mm]
  m^{\pm}_t&=\frac1n\sum_{i=1}^n \one\{b_i(t)=q_t\} \one\{v_i=\nu_{\pm}\} \label{eq:m_pm}\\[2mm]
  \frac{dq_t}{dt}&= u_t  \,  {\cdot \one\{q_t\in(0,1)\}  } \label{eq:dq_dt}\\[2mm]
 \frac {d b_i(t) }{dt}\,  &=
   \begin{cases}
     +v_i& \text{for }   \; b_i(t)<q_t\\
     u_t&  \text{for }   \; b_i(t)=q_t \text{ and } |u_t|\le v_i\\
          \sign (u_t)v_i &  \text{for }   \; b_i(t)=q_t \text{ and } |u_t|> v_i\\
     -v_i&  \text{for }   \; b_i(t)>q_t\, .
   \end{cases}\label{eq:db_dt}\\[2mm]
  (q_0, B_0)&=(b_{np}, B),\qquad \text{(initial condition)}.\label{eq:init_cond}
\end{align}

We notice that $h_t$ is the {harmonic average} of the velocities of bids at the max-price.

Define the set of collision times, i.e. times at which a bid hits the max-price curve, 
\begin{equation}\label{Coll}
\cC:=\left\{t\ge 0: \:   \exists \, i\in\{1, \ldots,n\} \: \text{s.t. } b_i(t^-) \neq q_t \;\text{and } b_i(t)=q_t\right\}.
\end{equation}
We remark that outside collision times, if there are slow particles on $q_t$, then $|u_t|\le \nua$, if instead there are only fast particles on $q_t$ then there can not be separations.
As a consequence, for $t\notin \cC$, 
 the max-price and bid velocities satisfy the simplified system of equations \eqref{eq:dq_dt_simplified}-\eqref{eq:db_dt_simplified},
 \begin{align}\label{eq:dq_dt_simplified}
   \frac{dq_t}{dt}&= u_t  \, {\cdot \one\{q_t\in(0,1)\}  }
 \\[2mm]
u_t&=\frac{w_t-\ell_t}{w_t+\ell_t} \cdot h_t\,,\qquad t\notin \cC,\\[2mm]
 \frac {d b_i(t) }{dt}\,  &=
   \begin{cases}
     +v_i& \text{for }   \; b_i(t)<q_t\\
     u_t&  \text{for }   \; b_i(t)=q_t \\
     -v_i&  \text{for }   \; b_i(t)>q_t\, 
   \end{cases},\qquad t\notin \cC.\label{eq:db_dt_simplified}
\end{align}
 In other words, outside collision times, the max-price curve has a time-derivative that is equal to the one that it would have in the single-velocity case times the harmonic-average function. Bidders that are on $q_t$ travel along the max-price curve.
Whereas bidders that are not on the max-price curve, travel ballistically at velocity $\pm v_i$ depending on whether they are above or below $q_t$. 
All this holds up to the next collision-time.

On the other hand, when the $i$-th bidder is on the max-price curve there are two possibilities depending on the value of $u_t$ with respect to $v_i$: (a) if $v_i\ge |u_t|$ the bidder will keep traveling together with the max-price, taking its velocity up to the next collision-time, and (b) if $v_i< |u_t|$ the bidder velocity is instantaneously deviated from the max-price and assumes the value $\sign (u_t)v_i$. The bidder then {\em separates} from the max-price and travels ballistically up to its next collision.

\vskip.3cm

\begin{lemma}\label{finite}
The total number of collisions is finite, i.e. $|\cC|<\infty$.
\end{lemma}
\begin{proof}
Each time one of the fast bidders has a collision, it is absorbed at the max-price forever. For what concerns the slow bidders, each of them can have at most two collisions during each interval between collisions involving fast bidders. This concludes the proof.
\end{proof}
Lemma \ref{finite} implies that the dynamics is well defined, as stated in the following proposition.
\begin{proposition}
  \label{4}
Given $p\in[0,1]$ and an initial bid-velocity multi-set $B$, the system \eqref{eq:h_t}-\eqref{eq:init_cond} has a unique solution, denoted $\cT_tB$ with max-price $q_t$ for each $t\ge 0$.
\end{proposition}

\subsection{Convergence of the discrete model}
Here we show that the continuous model can be obtained as the limit of the $\gamma$ discrete model,  as $\gamma$ tends to $0$.

Let $B = [(b_1,v_1),\dots,(b_n,v_n)]$ be a  bid-velocity multi-set, and $B^{\gamma}$ its $\gamma$ discretized version,
 \begin{align}
    \label{eq:8}
 B^{\gamma} := [(b^\gamma_1, v_1), \dots, (b^\gamma_n,v_n)], \qquad b^\gamma_i:= \gamma \lfloor b_i/\gamma\rfloor.
  \end{align}
Recall the discrete operator $\rmT^{\gamma}$, defined in \eqref{tb1}-\eqref{6}, and denote the $\gamma$ rescaled operator by
  \begin{align}
\label{3}
    \rmT^\gamma_t:= (\rmT^\gamma)^{\lfloor t/\gamma\rfloor}. 
  \end{align}
  
  \begin{theorem}\label{conv}
Let $p\in [0,1]$ with $pn \in \mathbb N$. Let $B$ a bid-velocity multi-set with $n$ bids and $B^{\gamma}$ its $\gamma$-discretized version. Let $\cT_t$ be the continuous operator defined in Proposition \ref{4}, and $\rmT^\gamma_t$ the rescaled discrete operator defined in \eqref{3}. Denote by $q_t, b_i(t)$ and $q^\gamma_t,b_i^{\gamma}(t)$ the max-price and $i$-th bid of $\cT_t B$ and $\rmT_t^{\gamma}B^{\gamma}$, respectively.
Then
\begin{align}
\sup_{\gamma\in (0,1)} \sup_{t\ge 0} \frac 1 \gamma |b_i^{\gamma}(t)-b_i(t)|&<\infty \quad \text{for all } i=1, \ldots, n,\\
\sup_{\gamma\in (0,1)} \sup_{t\ge 0} \frac 1 \gamma |q_t^\gamma-q_t|&<\infty.
\end{align}
\end{theorem}
\begin{proof}
We prove the theorem using an inductive argument over the time intervals between two consecutive collisions. More precisely we denote by $0:=t_0\le t_1 \le t_2 \le \ldots \le t_m$, for some $m\in \mathbb N$, the collision times in the continuous dynamics. In other words these are  the elements of the set $\cC$ defined in \eqref{Coll},  that, from Lemma \ref{finite}, we know to be finite. 

\noindent
We prove, by induction on $l\in \{0, 1, \ldots, m-1\}$, that if 
\begin{equation}\label{111}
\sup_{\gamma\in (0,1)} \frac 1 \gamma |b_i^{\gamma}(t_l)-b_i(t_l)|<\infty \qquad \text{for all } i=1, \ldots, n
\end{equation}
and
\begin{equation}\label{222}
\sup_{\gamma\in (0,1)} \frac 1 \gamma |q_{t_l}^\gamma-q_{t_l}|<\infty,
\end{equation}
then 
\begin{equation}\label{1111}
\sup_{\gamma\in (0,1)} \sup_{t\in [t_l, t_{l+1}]} \frac 1 \gamma |b_i^{\gamma}(t)-b_i(t)|<\infty \qquad \text{for all } i=1, \ldots, n
\end{equation}
and
\begin{equation}\label{2222}
\sup_{\gamma\in (0,1)} \sup_{t\in [t_l, t_{l+1}]}  \frac 1 \gamma |q_t^\gamma-q_t|<\infty.
\end{equation}
For simplicity we prove that the implication holds true in the first inter-collision time interval $[0,t_1]$, being the proof for the subsequent intervals totally analogous.
We notice that, from the definition of $B^{\gamma}$ given in \eqref{eq:8},  the induction hypothesis \eqref{111} and \eqref{222} are satisfied at time $t_0=0$. 
\noindent
Define the set $J:=\{j\in \{1,\ldots, n\}: b_j=q_0\}$ and  $\tilde B^{\gamma}$ the multi-set $[(\gamma \lfloor b_j/\gamma\rfloor, v_j):j\in J]$, i.e. the bids that are at the max-price at time 0. Then consider its evolution: ${\rmT}^{\gamma k}\tilde B^{\gamma}$, for $k=\lfloor t/\gamma\rfloor$, and denote by $\tilde q_t^\gamma$ the associated max-price. 
Denote by $t_1$ the first collision time for the continuous dynamics, $t_1:=\inf\{t: \:t \in \cC\}$,  and by $t_1^{\gamma}$ the first collision time for the discrete dynamics, i.e. 
$$t_1^{\gamma}:=\sup \{t\ge 0: \:  \text{sgn}(b^{\gamma}_i(t)-q^{\gamma}_t)= \text{sgn}(b^{\gamma}_i(0)-q^{\gamma}_0) \:\:\forall i\notin J \: \}.$$

We notice that from the inductive hypothesis \eqref{111} and \eqref{222}, that we know to be satisfied at time $0$, it follows that, for the sub-configuration $\tilde B^{\gamma}$ also the hypothesis \eqref{ottantuno} and \eqref{q0} of Theorem \ref{ball1} are satisfied. In particular the statement \eqref{222} coincides with \eqref{q0}, whereas the condition \eqref{111} restricted to the $i\in J$ implies \eqref{ottantuno} for the  sub-configuration $\tilde B^{\gamma}$.

We can apply Theorem \ref{ball1} to $(\tilde B^{\gamma}_t,\tilde q_t^\gamma)$ up to the first (discrete or continuous) collision time to deduce that there exists $C>0$ such that
\begin{equation}\label{aa}
\sup_{\gamma\in (0,1)} \sup_{0\le t \le \min\{t_1,t^{\gamma}_1\} } \frac 1 \gamma |\tilde q_t^\gamma-q_t|\le C.
\end{equation}
On the other hand,   we notice that 
\begin{equation}\label{bb}
\tilde q^{\gamma}_t= q^{\gamma}_t \qquad \text{for all } 0\le t \le t_1^{\gamma}
\end{equation}
We distinguish two cases.

\noindent
For all $\gamma\in (0,1)$ such that $t_1^{\gamma} \ge t_1$, from \eqref{aa} and \eqref{bb} it follows that
 \begin{equation}
 \sup_{0\le t \le t_1 } \frac 1 \gamma | q_t^\gamma-q_t|<C.
\end{equation}
Let now $\gamma $ be such that $t_1^{\gamma} < t_1$,  from \eqref{aa} and \eqref{bb} 
we have that
 \begin{equation}\label{aaa}
 \sup_{0\le t \le t_1^{\gamma} } \frac 1 \gamma | q_t^\gamma-q_t|<C.
\end{equation}
From the definition of $t_1^{\gamma}$, we know that there exists  $i^*\in \{1, \ldots, n\}\setminus J$ such that $\text{sgn}(b_{i^*}(t^{\gamma}+\gamma)-q^{\gamma}(t^{\gamma}+\gamma))\neq\text{sgn}(b_{i^*}(0)-q^{\gamma}_0)$, i.e. $i^*$ is the label of the first particle (in the discrete configuration) hitting the max-price after $t_1^{\gamma}$.
Then, necessarily 
\begin{equation}\label{aa1}
|q^{\gamma}(t_1^{\gamma})-b_{i^*}^{\gamma}(t_1^{\gamma})|\le 2 \nub \gamma, 
\end{equation}
For $\gamma$ small enough we can assume that $i^*$ coincides with its continuous counterpart, i.e. that it is the value such  that $b_{i^*}(t_1^-)\neq q(t_1^-)$ and $b_{i^*}(t_1)=q(t_1)$. 
Then, since at time 0
\begin{equation}
|b_{i^*}(0)-b_{i^*}^{\gamma}(0)|\le \gamma, 
\end{equation}
then
\begin{equation}\label{bbb}
|b_{i^*}(t_1^{\gamma})-b_{i^*}^{\gamma}(t_1^{\gamma})|\le  \gamma, 
\end{equation}
then using \eqref{aa1}, \eqref{aaa} and \eqref{bbb} it follows that
\begin{equation}\label{c}
|b_{i^*}(t_1^{\gamma})-q(t_1^{\gamma})|\le  (C+2\nub+1) \gamma
\end{equation}
then, since for all $t\in [t_1^{\gamma}, t_1]$ both $b_{i^*}(t)$ and $q(t)$ behave linearly, it follows that there exists a $C'>0$
\begin{equation}\label{c1}
|t_1^{\gamma}-t_1|\le C' \gamma.
\end{equation}
Then the number of discrete time-steps that are in the interval $[t_1^{\gamma},t_1]$ is smaller-equal than $C'$, then, necessarily
 \begin{equation}
 \sup_{t_1^{\gamma}\le t \le t_1 } \frac 1 \gamma | q_t^\gamma-q_t|<C+2\nub C':=C_1.
\end{equation}
Then we can conclude that
 \begin{equation}\label{A}
\sup_{\gamma\in (0,1)} \sup_{0 \le t \le t_1 } \frac 1 \gamma | q_t^\gamma-q_t|<C_1.
\end{equation}
We now want to prove that 
\begin{equation}
\sup_{\gamma\in (0,1)} \sup_{0\le t \le t_1} \frac 1 \gamma |b_i^{\gamma}(t)-b_i(t)|<\infty \qquad \text{for all } i=1, \ldots, n
\end{equation}
There are three possible cases.
\begin{itemize}
\item
If  $i\notin J$  (i.e. $b_i(0)\neq q_0$) and $b_i(t_1)\neq q_{t_1}$, then it means that for all $t\in [0, t_1]$ the continuous bid  travels ballistically  without interacting with the max-price. The same holds true for the discrete bid   if we take $\gamma$ small enough.  The two bids travel at the same velocity therefore the bound is easily verified.
\item
If  $i\notin J$ (i.e.  $b_i(0)\neq q_0$) and $b_i(t_1)= q_{t_1}$, it means that $i=i^*$ and the bound follows from the fact that there is no interaction with the max-price until $\min \lbrace t_1, t_1^{\gamma}\rbrace$ and the bound \eqref{c1}.
\item
If $i \in J$ (i.e. $b_i(0)= q_0$) we can apply the statement \eqref{ping}-\eqref{pong} of Theorem \ref{ball1} up to the first collision time to deduce that
\begin{equation}\label{bongo}
\sup_{\gamma\in (0,1)} \sup_{0\le t \le t_1 \wedge t_1^{\gamma}} \frac 1 \gamma |b_i^{\gamma}(t)-b_i(t)|<\infty \qquad \text{for all } i=1, \ldots, n.
\end{equation}
If $t_1\le t_1^{\gamma}$, then the bound is proven. If $\gamma$ is such that $t_1\ge t_1^{\gamma}$   then we use \eqref{bongo} up to time $t_1^{\gamma}$. For  $t \in [t_1^{\gamma},  t_1 ]$  we use  \eqref{c1} and the fact that the distance $|b_i^{\gamma}(t)-b_i(t)|$ remains bounded.
\end{itemize}
So \eqref{1111} and \eqref{2222} are proven for $l=0$. In particular the induction hypothesis  \eqref{111} and \eqref{222}  are satisfied at time $t_1$ i.e. the beginning of the next inter-collision interval $[t_1, t_2]$. This concludes the proof of the induction step and then of the Theorem.
\end{proof}

\subsection{Multi velocity bids and Chentsov colliding surfaces}\label{s7}

Assume that the bids may have two velocities. We describe the system as two colliding Chentsov surfaces. More precisely, let $\mu$ be a probability measure on $[0,1]\times\{\nu_-, \nu_+\}$ given by
\[
\mu(db, dv) = \alpha^- \, \mu^-(db) \, \delta_{\nu_-}(dv) \;+\; \alpha^+ \, \mu^+(db) \, \delta_{\nu_+}(dv),
\]
for some $\alpha^\pm \ge 0$ with $\alpha^- + \alpha^+ = 1$, where $\mu^\pm$ are probability measures on $[0,1]$. This is a mixture of measures $\mu^\pm$ associated to velocities $\nu_\pm$. 

For $\epsilon > 0$, let $\sB^\epsilon$ be a set of $\vep^{-1}$ independent bid-velocity points $(b_i,v_i)$ in $[0,1]\times\{\nu_-, \nu_+\}$ with distribution $\mu$.  
For each realization of $\sB^\epsilon$, denote by $\mathcal{T}_t \sB^\epsilon$ the unique solution of the multi-velocity collision system \eqref{eq:h_t}-\eqref{eq:init_cond} with initial condition $\sB^\epsilon$ (existence and uniqueness are guaranteed by Proposition~\ref{4}).

Define the empirical cumulative distribution function at time $t$ by:
\[
H^\epsilon(x,t):=F^\epsilon_t(x) := \epsilon \sum_{(b,v) \in \mathcal{T}_t \sB^\epsilon} \mathbf{1}_{\{b \le x\}}.
\]
The function $H^\epsilon:[0,1]\times \RR^+\to [0,1]$ can be seen as two colliding Chentsov surfaces with a jump along the trajectory of the max-price $q_t$. See Figures \ref{fig:gradini_2d} and \ref{fig:gradini_3d}.

We denote by $H(x,t):=\EE H^\epsilon(x,t)$ the deterministic field associated to $\mu$; it can be seen as the cumulative distribution function at time $t$, that is, the push-forward profile,
\[
H(x,t) = (\mathcal{T}_t \mu) ([0,x] \times \{\nu_-, \nu_+\}),
\]
where $\mathcal{T}_t \mu$ represents the deterministic continuous transport of the initial background measure $\mu$ under the multi-velocity dynamics \eqref{eq:h_t}-\eqref{eq:init_cond}.

\begin{remark}[Chentsov surfaces and structure]\rm
Call $H^\epsilon_p$ the random surface associated with a winner mass $p \in [0,1]$.
When $p=0$, all bid trajectories point downwards and are absorbed at the boundary $x=0$. When $p=1$, trajectories point upwards and are absorbed at $x=1$. In these boundary cases, the random surfaces $H^\epsilon_0$ and $H^\epsilon_1$ are standard Chentsov fields associated with the non-interacting processes $\mathcal{T}_t \sB^\epsilon$ (see \cite{ffgs23}).

For $p \in (0,1)$, under the regime $|u_t| \le \nu_-$ (i.e., when no slow-particle separation occurs along $q_t$), the field $H^\epsilon_p$ coincides with $H^\epsilon_0$ above the max-price trajectory $(q^\epsilon_t)_{t\ge0}$ and with $H^\epsilon_1$ below it, where $q^\epsilon_t := q(\mathcal{T}_t \sB^\epsilon)$. Along the line $q^\epsilon_t$, there is a spatial jump of height
\[
H^\vep (q^\epsilon_t,t)-H^\vep ((q^\epsilon_t)^-,t)= F^\epsilon_t(q^\epsilon_t) - F^\epsilon_t((q^\epsilon_t)^-),
\]
representing the total mass of bids attached to the max-price at time $t$.
\end{remark}

\begin{figure}[htbp]
                               \centering
       \includegraphics[width=.6\textwidth]{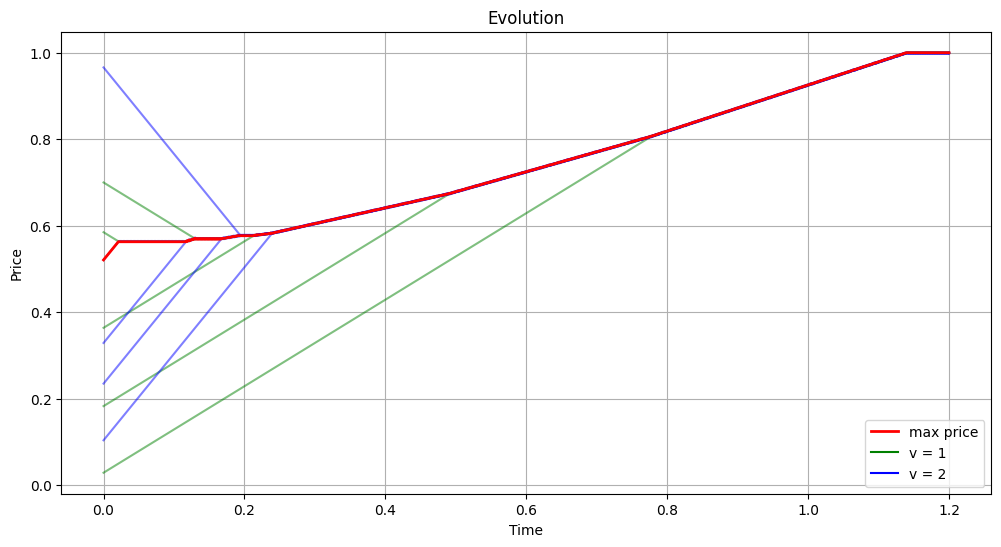} 
          \vspace{.5cm}
        \caption{Discrete auction evolution with $n=10$ bids, $n_- = 5$ have $\nu_{-} = 1$ and  $n_+ = 5$ have $\nu_{+} = 2$ with $p=0.7$ and $\gamma = 0.01$.}
        \label{fig:gradini_2d}
      \end{figure}
 \begin{figure}[htbp]
        \centering
        \includegraphics[width=.6\textwidth]{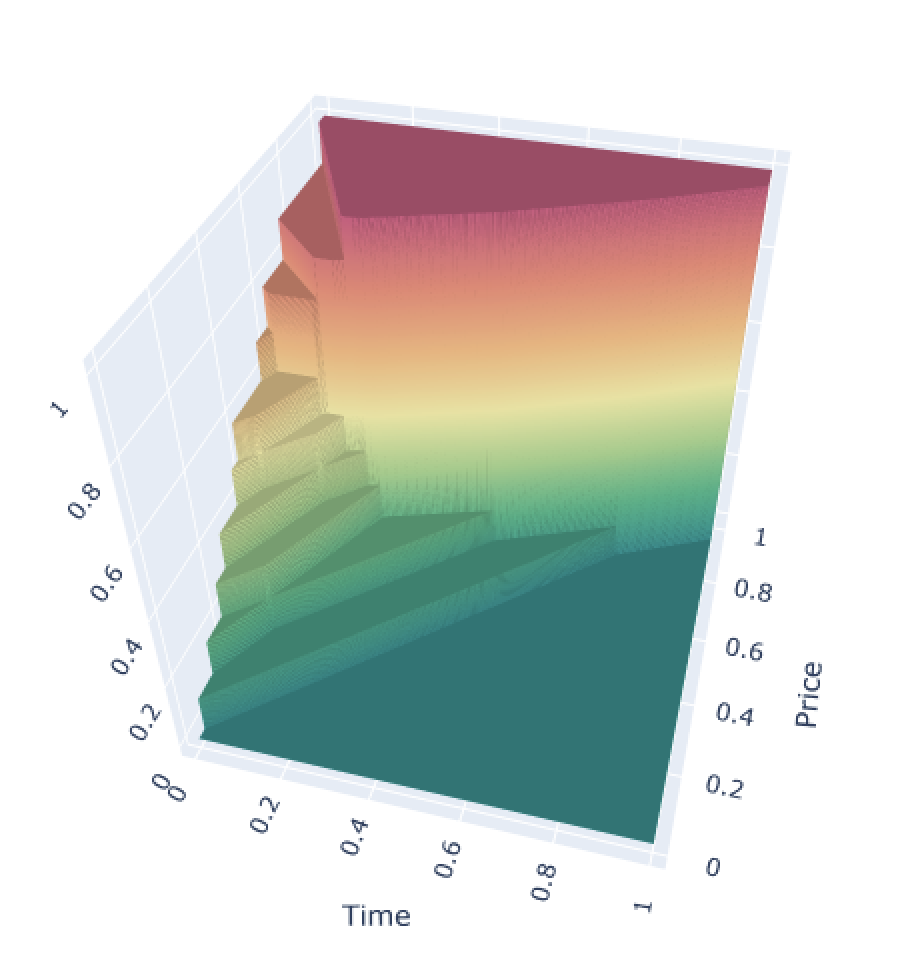} 
        \caption{Chentsov colliding fields $H^\epsilon$ corresponding to the evolution $T_t\sB^\epsilon$ in Fig.\/\ref{fig:gradini_2d}.}
        \label{fig:gradini_3d}
\end{figure}

\begin{definition}[Collision times for a realization]
\label{def:collision_times}
For a given realization of $\sB^\epsilon$, let $\mathcal{C}(\sB^\epsilon)$ denote the set of collision times for the multi-velocity dynamics \eqref{eq:h_t}-\eqref{eq:init_cond}:
\[
\mathcal{C}(\sB^\epsilon) := \left\{ t \ge 0 : \exists \, i \text{ s.t. } b_i(t^-) \neq q_t \text{ and } b_i(t) = q_t \right\},
\]
where $(b_i(t))_{i}$ are the positions of the bids evolving according to \eqref{eq:h_t}-\eqref{eq:init_cond}. 
\end{definition}

\begin{lemma}[Finite number of collisions]
\label{lem:finite_collisions}
For every fixed $\epsilon > 0$, and for almost every realization of $\sB^\epsilon$, the set $\mathcal{C}(\sB^\epsilon)$ is finite.
Consequently, for each $t \ge 0$, the set $\mathcal{C}_\epsilon(t) := \mathcal{C}(\sB^\epsilon) \cap [0,t]$ is finite almost surely, with cardinality $N_\epsilon(t) := |\mathcal{C}_\epsilon(t)| < \infty$.
\end{lemma}

\begin{proof}
Since $\mu$ is a finite measure on $[0,1] \times \{\nu_-, \nu_+\}$, a realization of $\sB^\epsilon$ contains a finite number $N_\epsilon = N^-_\epsilon + N^+_\epsilon$ of total bids almost surely.
By \eqref{eq:u_t_def}, the max-price speed satisfies $|u_t| \le \nu_+$. Therefore, fast bids ($v = \nu_+$) cannot separate once they reach $q_t$; they are absorbed at their first collision and never collide again.
Conversely, each slow bid ($v = \nu_-$) can undergo at most two collisions (one upon joining $q_t$ and at most one upon separating from it).
Hence, the total number of collisions is deterministically bounded by $|\mathcal{C}(\sB^\epsilon)| \le 2 N^-_\epsilon + N^+_\epsilon < \infty$ almost surely.
\end{proof}

\subsubsection{Law of Large Numbers}

We now state the Law of Large Numbers (LLN) for the field $H^\epsilon(x,t)$ as $\epsilon \to 0$.

\begin{theorem}[LLN for the colliding Chentsov fields]
\label{thm:lln_pointwise}
For every fixed $t \in [0,\infty)$, the field $H^\epsilon(\cdot,t)$ converges uniformly to its mean $H(\cdot, t)$ almost surely as $\epsilon \to 0$, i.e.,
\begin{align}
  \label{101}
  \lim_{\epsilon \to 0} \sup_{x\in[0,1]} \left| H^\epsilon(x,t) - H(x,t) \right| = 0 \quad \text{a.s.}
\end{align}
\end{theorem}
The limit \eqref{101} is also called \emph{fluid limit}.
\begin{proof}
At time $t=0$, the initial empirical profile $H^\epsilon$ is generated by a point process $\sB^\vep$. 
By the functional Law of Large Numbers for iid,see, e.g. \cite{vdVaartWellner1996}, 
we have 
\[
\sup_{x\in[0,1]} \left|H^\epsilon(x,0) - H(x,0) \right| \xrightarrow{\epsilon \to 0} 0 \quad \text{a.s.}
\]
For any time $t > 0$, the transport map $\mathcal{T}_t$ acts as a piecewise-isometric, continuous transformation on the space of bid measures endowed with the $L^1$-distance on their distribution functions (or the 1-Wasserstein metric). Specifically, since all particle speeds are bounded above by $\nu_+$ and the flow $\mathcal{T}_t$ preserves measure order, the mapping $H(\cdot,0) \mapsto H(\cdot,t)=:\mathcal{T}_tH(\cdot,0) $ is continuous with respect to the uniform norm $\|\cdot\|_\infty$.

Therefore, applying the continuous mapping principle for the flow $\mathcal{T}_t$, the uniform convergence of the initial profile $H^\epsilon(\cdot,t) \to H(\cdot, 0)$ propagates to any finite time $t > 0$:
\[
\sup_{x\in[0,1]} \left|H^\epsilon(x,t) - H(x,t) \right| = \sup_{x\in[0,1]} \left| (\mathcal{T}_t H^\epsilon(x,0) - (\mathcal{T}_t H(x,0)) \right| \xrightarrow{\epsilon \to 0} 0 \quad \text{a.s.}
\]
This completes the proof.
\end{proof}


\subsubsection{Gaussian Fluctuations for the Chentsov Field: Conjecture}

Having established the pointwise Law of Large Numbers for the Chentsov surface $H^\epsilon(x,t)$, we turn to the fluctuation process. Define the space-time empirical fluctuation field by
\begin{equation}
\label{eq:chentsov_fluct}
\zeta^{H^\epsilon}(x,t) := \epsilon^{-1/2} \left( H^\epsilon(x,t) - H(x,t) \right), \qquad (x,t) \in [0,1] \times [0,T].
\end{equation}

For $p=0$ and $p=1$, $(\zeta^{H^\epsilon}(x,t): (x,t)\ne (q_t,t))$ is a Chentsov fluctuation process, generated by non-interacting spatial transport, whose weak convergence in $D([0,1]\times[0,T])$ to a centered Gaussian process $\zeta^H$ with covariances
\begin{align}
  \Cov(\zeta^H(x,t),\zeta^H(z,s))=
  \begin{cases}
    0 & \text{if the segment $ \overline{(x,t)(z,s)}$ crosses the line $q_t$}\\
    \mu((x,t)(z,s)) & \text{else}.
  \end{cases}
\end{align}
where$(x,t)(z,s)$ is defined as the set of lines crossing the 
the Lévy-Chentsov field conditioned to be zero at (a transported Brownian sheet) is a standard result for independent Poisson empirical transport (see, e.g., \cite{ffgs23,vdVaartWellner1996}). 
For $p \in (0,1)$, however, the interaction at the max-price trajectory $q_t$ introduces a dynamic boundary where particles aggregate, creating a spatial discontinuity in the fluid limit $H(x,t)$ along the curve $x = q_t$.

Based on the Poisson initial structure and the local linearity of the multi-velocity flow $\mathcal{T}_t$ away from collisions, we formulate the following conjecture:

\begin{conjecture}[Limit Fluctuation Field]
\label{conj:covariance_chentsov}
Let $\zeta^F(x,t)$ denote the centered Gaussian limit process of the empirical fluctuation field 
\[
\zeta^{F^\epsilon}(x,t) = \epsilon^{-1/2} \left( F^\epsilon_t(x) - H(x,t) \right)
\]
on $[0,1] \times [0,T] \setminus \mathcal{C}$, with covariance
    \[
    \operatorname{Cov}\left( \zeta^F(x,s), \zeta^F(y,t) \right) = \mu \left( \mathcal{T}_s^{-1}\big([0,x]\big) \cap \mathcal{T}_t^{-1}\big([0,y]\big) \right).
    \]
for all $(x,s)$, $(y,t)$ with $0 \le s \le t \le T$ and $s, t \notin \mathcal{C}$.
Here $\mathcal{T}_t$ is the deterministic transport dynamics defined by \eqref{eq:h_t}-\eqref{eq:init_cond}, and $\mu$ is the initial spatial intensity measure of the Poisson process.
\end{conjecture}


\vskip.2cm
\noindent
{\bf Acknowledgments.} 
We acknowledge financial support under the National Recovery and Resilience Plan (NRRP), Mission 4, Component 2, Investment 1.1, Call for tender No. 104 published on 02.02.2022 by the MUR, funded by the European Union, NextGenerationEU,  Project n. 202277WX43, CUP E53D23005500006.
P.A.F. acknowledges the hospitality of the University of Modena and Reggio Emilia. G.C., C.F., and N.M. acknowledge the hospitality of the University of Buenos Aires. G.C. and C.F. are members of the Gruppo Nazionale per la Fisica Matematica of the Istituto Nazionale di Alta Matematica (INdAM) and acknowledge support from the FAR UniMoRe project, CUP E93C25002460005.

\bibliographystyle{plain}
\bibliography{auctions}

\bigskip
\parskip 0pt
\noindent Gioia Carinci, Chiara Franceschini, Nicola Manelli 

\noindent {Department of Physics, Informatics and Mathematics, University of Modena and Reggio Emilia,  Modena, Italy}

\noindent \href{mailto: gioia.carinci@unimore.it}{\texttt{gioia.carinci@unimore.it}}, \href{mailto: cfrances@unimore.it}{\texttt{cfrances@unimore.it}}, \href{mailto:nicola.manelli@unimore.it}{\texttt{nicola.manelli@unimore.it}}

\bigskip

\noindent Pablo A. Ferrari

\noindent {Departamento de Matemática, Facultad de Ciencias Exactas y Naturales, Universidad de Buenos Aires and IMAS-CONICET, Buenos Aires, Argentina}

\noindent \href{mailto:pferrari@dm.uba.ar}{\texttt{pferrari@dm.uba.ar}}, \href{https://mate.dm.uba.ar/\string~pferrari/}{\texttt{https://mate.dm.uba.ar/\string~pferrari}}

\end{document}